\documentclass[a4paper,11pt,twoside]{amsart}

\usepackage{amsmath}
\usepackage{amsthm}
\usepackage{amsfonts, amssymb}
\usepackage{mathrsfs}
\usepackage[all]{xy}
\usepackage{latexsym}
\usepackage{graphicx}
\usepackage{enumerate}
\usepackage{url}

\usepackage{tikz}

\usepackage[labelformat=simple]{subcaption}

\usetikzlibrary{arrows}
\usetikzlibrary{decorations.markings}

\theoremstyle{plain}
\newtheorem{lemma}{Lemma}[section]
\newtheorem{prop}[lemma]{Proposition}
\newtheorem{theo}[lemma]{Theorem}
\newtheorem{coro}[lemma]{Corollary}
\theoremstyle{remark}
\newtheorem{rem}[lemma]{Remark}
\newtheorem*{notat}{Notation}
\theoremstyle{definition}
\newtheorem{definition}[lemma]{Definition}
\newtheorem{ex}[lemma]{Example}
\newtheorem{question}[lemma]{Open question}
\newtheorem{conjecture}[lemma]{Conjecture}

\newcommand{\op}{\text{op}}
\newcommand{\obj}{\mathrm{Obj}}
\newcommand{\mor}{\mathrm{Mor}}
\newcommand{\Hom}{\mathrm{Hom}}
\newcommand{\aut}{\mathrm{Aut}}
\newcommand{\id}{\mathrm{Id}}
\newcommand{\C}{\mathscr{C}}
\newcommand{\D}{\mathscr{D}}
\newcommand{\sgn}{\mathrm{sgn}}
\newcommand{\Spl[1]}{\textnormal{(S#1)}}
\newcommand{\N}{\mathbb{N}}
\newcommand{\Z}{\mathbb{Z}}
\newcommand{\R}{\mathbb{R}}
\newcommand{\Pp}{\mathbb{P}}
\newcommand{\Su}{\mathbb{S}}
\newcommand{\rk}{\mathrm{rk}}
\newcommand{\tq}{\;/\;}
\newcommand{\X}{\mathcal{X}}
\newcommand{\K}{\mathcal{K}}
\newcommand{\mxl}{\textnormal{mxl}}
\newcommand{\mnl}{\textnormal{mnl}}
\newcommand{\M}{\mathcal{M}_X}
\newcommand{\B}{\mathcal{B}_X}
\newcommand{\Relp}{\mathcal{R}_X}
\newcommand{\Relm}{\mathcal{S}_X}
\newcommand{\A}{\mathcal{A}}
\newcommand{\CC}{\mathcal{C}}
\newcommand{\Q}[1]{\mathcal{Q}_{#1}}
\newcommand{\T}{\mathbb{T}^2}
\newcommand{\Tn}{\mathbb{T}^n}
\newcommand{\Kl}{\mathbb{K}}
\newcommand{\im}{\text{Im}\:}
\newcommand{\arr}{-triangle 45}

\begin{document}

\title{Poset splitting and minimality of finite models}

\author{Nicol\'as Cianci}
\address{Facultad de Ciencias Exactas y Naturales \\ Universidad Nacional de Cuyo \\ Mendoza, Argentina.}
\email{nicocian@gmail.com}

\author{Miguel Ottina}
\email{mottina@fcen.uncu.edu.ar}


\thanks{Research partially supported by grant M015 of SeCTyP, UNCuyo.}

\begin{abstract}
We develop a novel technique, which we call poset splitting, that allows us to solve two open problems regarding minimality of finite models of spaces: the nonexistence of a finite model of the real projective plane with fewer than 13 points and the nonexistence of a finite model of the torus with fewer than 16 points. Indeed, we prove much stronger results from which we also obtain that there does not exist a finite model of the Klein bottle with fewer than 16 points and that the integral homology groups of finite spaces with fewer than 13 points are torsion-free, settling a conjecture of Hardie, Vermeulen and Witbooi. Furthermore, we also apply our technique to give a complete characterization of the minimal finite models of the real projective plane, the torus, and the Klein bottle.

In addition, we show that the poset splitting technique has an intrinsic interest giving original topological results that can be obtained from its application, such as a generalization of Hurewicz's theorem for non-simply-connected spaces and a generalization of a result of R. Brown on the fundamental group of a space.
\end{abstract}

\maketitle

\section{Introduction}

A well-known result of McCord \cite{McC} asserts that for every finite simplicial complex $K$ there exists a finite T$_0$--space $\mathcal{X}(K)$ and a weak homotopy equivalence $|K|\to\mathcal{X}(K)$. Hence, every compact CW-complex $X$ admits \emph{finite models}, that is, finite topological spaces which are weak equivalent to $X$. Moreover, the poset of cells of a finite and regular CW-complex $X$ is a finite model of $X$. Finite models permit the study of homotopical invariants of spaces by means of combinatorial tools and serve as a method to tackle open problems in algebraic topology from a different point of view \cite{BarLN}.

A natural problem regarding finite models is to find \emph{minimal finite models} for a given compact CW-complex $X$, that is finite models of $X$ of the minimum possible cardinality. The first question in this line was stated by J.P. May in \cite{May}, where he conjectures that the $n$--fold non-Hausdorff suspension of the $0$--sphere, denoted $\Su^n S^0$, is a minimal finite model of the $n$--sphere. This question was answered positively by Barmak and Minian in \cite{BM_minimal}, where they also prove that $\Su^n S^0$ is the only minimal finite model of the $n$--sphere. In the same article they give a characterization of the minimal finite models of finite graphs.

Despite the simplicity of its formulation, it turns out that finding minimal finite models of spaces is a very hard problem due to the lack of tools which relate homotopy invariants of a finite T$_0$--space with the combinatorial information present in its Hasse diagram. Also, as the number of (unlabeled) posets on $n$ points seems to grow exponentially on $n^2$ \cite{B-McKay} and since there does not exist an effective algorithm to describe posets of $n$ points, even a computer-aided brute force method becomes impractical for posets with more than 12 points. The intrinsic difficulty of the minimality problem is reflected in the fact that, prior to this article, the only known minimal finite models were those of the spheres, the graphs and the homotopically trivial spaces.

Several open questions about minimal finite models appear in the literature. For example, a finite model of the real projective plane of 13 points is exhibited in \cite[Example 7.1.1]{BarLN} as the poset of cells of a certain regular CW-complex structure and is also constructed in \cite{HVW} from a finite version of the mapping cone. Its minimality was raised as an open question in \cite{HVW} and in \cite{BarLN} and verified by M. Adamaszek with a computer-assisted proof which required the analysis of approximately $10^8$ cases \cite{Ad}.

In this article we give the first purely mathematical proof of this result showing that if $X$ is a finite and connected space with fewer than 13 points then $\pi_1(X)$ is a free group. Moreover, we prove that there exist (up to homeomorphism) exactly two minimal finite models of the real projective plane: the aforementioned one and its opposite.

In \cite{HVW}, Hardie, Vermeulen and Witbooi state a stronger conjecture: that the finite model of the real projective plane mentioned above is the smallest finite space with 2--torsion in any integral homology group. We will also prove this conjecture showing that if $X$ is a finite space such that $H_n(X)$ has torsion for some $n\in\N$ then $X$ has at least $13$ points.

Also, in section \ref{section_torus} we prove that there does not exist a finite model of the torus with fewer than 16 points settling another open problem \cite[p.44]{BarLN}. Indeed, we prove a stronger result from which we also derive the non-existence of a finite model of the Klein bottle with fewer than 16 points. Finite models of the Klein bottle had not been studied before. Furthermore, we prove that there exist (up to homeomorphism) exactly two minimal finite models of the torus: $\Su S^0 \times \Su S^0$  and the space $\T_{1,1}$ of page \pageref{fig_Q11} of this work. The first one was known, but not the second one, which was obtained from our methods. Moreover, in the same proof we also obtain that there exist (up to homeomorphism) exactly two minimal finite models of the Klein bottle: the spaces $\Kl_{1,0}$ and $\Kl_{0,1}$ of figures \ref{fig_Q10} and \ref{fig_Q01}. Remarkably, these four minimal finite models are posets of height 2 obtained by suitably combining two posets of height one in four possible ways, resembling the way the torus and the Klein bottle are obtained as a quotient of a cylinder.

All these original results are achieved by developing a novel technique which we call \emph{poset splitting} and which turns out to be interesting in itself. This technique is developed in section \ref{section_splitting} where we show some new topological results that can be obtained from it. In particular, we obtain a generalization of a result of R. Brown \cite{Bro} regarding the fundamental group and a generalization of Hurewicz's theorem for non-simply-connected spaces.

Finally, in the last section of this article we exhibit finite models for real projective spaces, tori and Moore spaces of type $(\Z_n,k)$ raising open questions about their minimality.

\section{Preliminaries}

In the first part of this section we will recall the basic notions of the theory of finite topological spaces, fix notation, and give some basic results which were not found in the classical literature. For a comprehensive exposition on finite spaces the reader may consult \cite{BarLN}. In the second part of this section we will review the categorical approach to homotopy theory of posets with some new insights which are related to works of C. Cibils and J. MacQuarrie \cite{CM} and of J. Barmak and G. Minian \cite{BM2}.

\subsection*{Classical definitions and results on homotopy theory of posets}

The theory of finite topological spaces is based in the well-known functorial correspondence between Alexandroff T$_0$--spaces and posets given by Alexandroff which preserves the underlying set and is induced by the inclusion relation between minimal open neighbourhoods \cite{Alex}. In this way, any Alexandroff T$_0$--space is naturally endowed with a partial order. Indeed, if $X$ is an Alexandroff T$_0$--space and $x\in X$ then the minimal open set which contains $x$ is $U_x=\{a\in X \tq a\leq x\}$ and the minimal closed set which contains $x$ is $F_x=\overline{\{x\}}=\{a\in X \tq a\geq x\}$. Under this correspondence, continuous map between Alexandroff T$_0$--spaces correspond to order-preserving morphisms between posets.

If $X$ is an Alexandroff T$_0$--space it is standard to denote $\widehat{U}_x=\{a\in X \tq a< x\}$, $\widehat{F}_x=\{a\in X \tq a> x\}$, $C_x=U_x\cup F_x$ and $\widehat{C}_x=C_x-\{x\}$. In the case that several topological spaces are considered at the same time, we will denote $U_x$ by $U_x^X$ to indicate the space in which the minimal open set is considered. We will use similar notations for $F_x$, $C_x$, $\widehat{U}_x$, $\widehat{F}_x$ and $\widehat{C}_x$. 

Stong gives in \cite{Sto} a simple algorithm to decide whether two finite T$_0$--spaces are homotopy equivalent or not. Following May's terminology \cite{May}, if $X$ is a finite T$_0$--space we will say that a point $x\in X$ is an \emph{up beat point} (resp. \emph{down beat point}) of $X$ if the subposet $\widehat{F}_x$ has a minimum (resp. if the subposet $\widehat{U}_x$ has a maximum). Stong proves in the aforementioned article that if $x$ is a beat point of $X$ then $X-\{x\}$ is a strong deformation retract of $X$ and that two finite T$_0$--spaces are homotopy equivalent if and only if one obtains homeomorphic spaces after successively removing their beat points. Using the results of Stong it is easy to prove that if $X$ is a finite T$_0$--space and $x\in X$ then $C_x$ is contractible.

A different approach to the study of finite topological spaces was given by McCord in \cite{McC} where he proves that the \emph{face poset} $\mathcal{X}(K)$ of a simplicial complex $K$ (which is the poset of simplices of $K$ ordered by inclusion) is weak homotopy equivalent to $|K|$. More generally, if $K$ is a regular CW-complex then the face poset $\mathcal{X}(K)$ of $K$ is weak homotopy equivalent to $K$. 

McCord also proves in \cite{McC} that if $X$ is an Alexandroff T$_0$--space then there exists a weak homotopy equivalence $|\mathcal{K}(X)| \to X$. Here, $\mathcal{K}(X)$ is the \emph{order complex} of $X$, that is, the simplicial complex of the non-empty chains of $X$ and $|\mathcal{K}(X)|$ denotes its geometric realization.

Note that $\mathcal{K}(\mathcal{X}(K))$ is the barycentric subdivision of $K$. This leads to the definition of the \emph{barycentric subdivision} of a poset $X$ as $X'=\mathcal{X}(\mathcal{K}(X))$. It follows that $X$ and $X'$ are weak homotopy equivalent.

McCord also shows that for every Alexandroff space $X$ there exists a quotient map $q:X\to Y$ such that $Y$ is an Alexandroff T$_0$--space and $q$ is a homotopy equivalence \cite[Theorem 4]{McC}. 

From the results of McCord it follows that the singular homology groups of an Alexandroff T$_0$--space $X$ are naturally isomorphic to the simplicial homology groups of $\K(X)$. Hence, the homology groups of an Alexandroff T$_0$--space $X$ can be described in terms of the chains of $X$.

This idea is exploited in \cite{CO} where we obtain a spectral sequence which converges to the integral homology groups of a locally finite T$_0$--space in which the differentials have an explicit and simple description. Under certain conditions, the first page of this spectral sequence reduces to a chain complex, which motivates the following definition.

\begin{definition}\label{def:quasi_relativo}
Let $X$ be a locally finite T$_0$--space and let $A\subseteq X$ be a subspace. We say that $(X,A)$ is a \emph{relative quasicellular pair} if $A$ is open in $X$ and there exists an order preserving map $\rho:X-A\longrightarrow \N_0$, which will be called \emph{quasicellular morphism for $(X,A)$}, such that
\begin{enumerate}[(1)]
\item The set $\{x\in X-A:\rho(x)=n\}$ is an antichain for every $n\in \N_0$. 
\item For every $x\in X-A$, the reduced integral homology of $\widehat{U}_x^X$ is concentrated in degree $\rho(x)-1$.
\end{enumerate}
We say that a locally finite $T_0$--space is \emph{quasicellular} if $(X,\varnothing)$ is a relative quasicellular pair.
\end{definition}

Note that if $(X,A)$ is a relative quasicellular pair, $\rho$ is a quasicellular morphism for $(X,A)$ and $x,y\in X-A$, then $x<y$ implies that $\rho(x)<\rho(y)$.

Note also that if $K$ is a regular CW-complex then $\X(K)$ is quasicellular. In particular, if $X$ is an Alexandroff T$_0$--space then $X'$ is quasicellular.

Hereafter, homology will mean integral homology and hence the group of coefficients will be omitted from the notation.

In \cite{CO} we obtained the following result, which will be applied in section \ref{section_splitting}.

\begin{prop} \label{coro_quasicel_relativo}
Let $(X,A)$ be a relative quasicellular pair and let $\rho$ be a quasicellular morphism for $(X,A)$. For each $n\in \N_0$, let $J_n=\{x\in X-A:\rho(x)=n\}$.

Let $C(X,A)=(C_n(X,A),d_n)_{n\in\Z}$ be the chain complex defined by
\begin{itemize}
\item $C_n(X,A)=\bigoplus\limits_{x\in J_n}\tilde{H}_{n-1}(\widehat{U}_x)$ for each $n\in \N_0$ and $C_n(X,A)=0$ for $n<0$.
\item For each $n\in \Z$, $d_n:\bigoplus\limits_{x\in D_p}\tilde{H}_{p+q-1}(\widehat{C}^{X_p}_x)\longrightarrow\bigoplus\limits_{y\in D_{p-1}}\!\!\tilde{H}_{p+q-2}(\widehat{C}^{X_{p-1}}_y)$ is the group homomorphism defined by
\[d_n\left(\left(\left[\sum\limits_{i=1}^{l_x}a^x_i s^x_i\right]\right)_{x\in D_n}\right)=\left(\left[\sum\limits_{x\in D_n}\sum\limits_{s^x_i\ni y}(-1)^{\#(s^x_i\cap \widehat U_y)}a^x_i(s^x_i-\{y\})\right]\right)_{y\in D_{n-1}}\]
where for every $x\in D_n$, $l_x\in \N$, and for every $i=\{1,\dots,l_x\}$, $a^x_i\in \Z$ and $s^x_i$ is an $(n-1)$--chain of $\widehat{U}_x$.
\end{itemize}
Then, $H_n(X,A)=H_n(C(X,A))$ for all $n\in\N_0$.
\end{prop}

McCord also defines the \emph{non-Hausdorff suspension} of a topological space $X$ as the space $\mathbb{S}(X)$ whose underlying set is $X\amalg \{+,-\}$ and whose open sets are those of $X$ together with $X\cup\{+\}$, $X\cup\{-\}$ and $X\cup\{+,-\}$, and proves that for every space $X$ there exists a weak homotopy equivalence between the suspension of $X$ and $\mathbb{S}(X)$ \cite{McC}.

Recall that if $X$ and $Y$ are finite T$_0$--spaces, the \emph{non-Hausdorff join} $X\circledast Y$ is defined as the poset whose underlying set is the disjoint union $X\amalg Y$ and whose partial order is defined keeping the partial orders within $X$ and $Y$ and setting $x\leq y$ for all $x\in X$ and $y\in Y$. Note that $\mathbb{S}(X)=X\circledast S^0$.

If $X$ is a poset, $X^\op$ will denote the poset $X$ with the inverse order and will be called the \emph{opposite} space of $X$. Note that $|\K(X^\op)|=|\K(X)|$ and hence $X$ and $X^\op$ are weak homotopy equivalent.

Finally, recall that a \emph{finite model} of a topological space $Z$ is a finite space which is weak homotopy equivalent to $Z$. For example, the $n$--fold non-Hausdorff suspension of the $0$--sphere, $\mathbb{S}^n S^0$, is a finite model of the $n$--sphere \cite{McC}. A \emph{minimal finite model} of a topological space $Z$ is a finite model of $Z$ of minimum cardinality. By \cite[Theorem 4]{McC}, minimal finite models are T$_0$--spaces.

We will give now some basic results that were not found in the literature and will be needed later.

\begin{prop} \label{prop_non_Hausdorff_suspension}
Let $X$ be a finite T$_0$--space and let $a,b\in X$. Then $U_a\cup U_b$ is homotopy equivalent to $\Su (U_a\cap U_b)$. In particular, if $U_{a}\cap U_{b}$ is contractible then $U_{a}\cup U_{b}$ is contractible.
\end{prop}

\begin{proof}
If $a\leq b$ or $b\leq a$ the result holds since $U_a\cup U_b$ and $\Su (U_a\cap U_b)$ are contractible. Hence, we may assume that $a$ and $b$ are incomparable.

Let $i:\{a,b\}\cup (U_{a}\cap U_{b})\to U_{a}\cup U_{b}$ be the inclusion map and let $r:U_a\cup U_b \to \{a,b\}\cup (U_{a}\cap U_{b})$ be defined by 
\begin{displaymath}
r(x) = \left\{
\begin{array}{cl}
x & \textnormal{ if $x\in U_a\cap U_b$} \\
a & \textnormal{ if $x\in U_a - U_b$} \\
b & \textnormal{ if $x\in U_b - U_a$}
\end{array}
\right .
\end{displaymath}
It is easy to check that the map $r$ is order-preserving and hence continuous. Also, $ri=\id$ and $ir\geq \id$. Thus, by \cite[Corollary 1.2.6]{BarLN}, $r$ is a strong deformation retraction.
\end{proof}

We ought to mention that the map $r$ of the previous proof already appears in the proof of proposition 11.2.3 of \cite{BarLN}. In a similar way, the following proposition is related to the notion of qc-reductions \cite[p.140]{BarLN}.

\begin{prop} \label{prop_weak_collapse}
Let $X$ be a finite T$_0$--space and let $a$ and $b$ be maximal elements of $X$. If $U_{a}\cup U_{b}$ is homotopically trivial then the quotient map $q:X\to X/\{a,b\}$ is a weak homotopy equivalence.
\end{prop}

\begin{proof}
Note that $X/\{a,b\}$ is a finite T$_0$--space (\cite[Proposition 2.7.8]{BarLN}). The result then follows from McCord's theorem (\cite[Theorem 6]{McC}) taking the basis of minimal open sets of $X/\{a,b\}$.
\end{proof}

The following simple lemma constitutes a key step for the results of this article.

\begin{lemma} \label{lemma_subdiv_X_A}
Let $X$ be an Alexandroff T$_0$--space and let $A \subsetneq X$. Then
\begin{enumerate}
\item $A'$ is an open subset of $X'$ and $(X-A)'\subseteq X'-A'$.
\item The inclusion $i:(X-A)'\to X'-A'$ is a weak homotopy equivalence.
\end{enumerate}
\end{lemma}

\begin{proof}\ 

(1) Follows easily from the definition of barycentric subdivision of a poset.

(2) Let $\sigma\in X'-A'$ and let $\eta=\sigma\cap(X-A)$. Then $\eta\in (X-A)'$ and
\[i^{-1}(U_\sigma)=\{\tau\tq\textnormal{$\tau$ is a chain of $X-A$ and $\tau\subseteq \sigma$}\}=U_\eta\]
which is contractible. Hence, by McCord's theorem (\cite[theorem 6]{McC}), $i$ is a weak homotopy equivalence.
\end{proof}

\subsection*{Categorical approach to homotopy theory of posets}

When considered as a poset, an Alexandroff T$_0$--space $X$ can be regarded as a small category in the standard way, that is, the objects of the category are the elements of $X$ and for $x,y\in X$ there is a (unique) morphism from $x$ to $y$ if and only if $x\leq y$. It is well known that the nerve of this category is the simplicial set associated to the simplicial complex $\K(X)$ and hence the classifying space of $X$ is homeomorphic to $|\K(X)|$. In particular, the categorical definition of homotopy and homology groups of $X$ agrees with the usual definition for topological spaces.

Given a small category $\C$ we can consider its localization with respect to all its morphisms $\C[\mor(\C)^{-1}]$. We will write $L\C$ for this category, which is clearly a groupoid, and $\iota_\C:\C\to L\C$ for the natural localization functor. Note that a functor $F:\C\to \D$ between small categories induces a unique groupoid morphism $LF:L\C\to L\D$ in a functorial way \cite{GZ}.

D. Quillen proved in \cite{Qui} that the fundamental group of a small category $\C$ at an object $c_0$ is canonically isomorphic to $\aut_{L\C}(c_0)$, the group of automorphisms of $c_0$ in the groupoid $L\C$. So we can canonically identify the fundamental group of an Alexandroff T$_0$--space $X$ at $x_0$ with the group of automorphisms of $x_0$ in $LX$.

Now let $\mathcal{G}$ be a groupoid and let $\mathcal{N}$ be a subgroupoid of $\mathcal{G}$. Following \cite{Bro70}, we say that $\mathcal{N}$ is a normal subgroupoid of $\mathcal{G}$ if $\obj(\mathcal{N})=\obj(\mathcal{G})$ and for all $x,y\in \obj(\mathcal{G})$ and $g\in \Hom_\mathcal{G}(x,y)$ we have that $g^{-1}\aut_\mathcal{N}(y)g=\aut_\mathcal{N}(x)$. In this case we can identify objects $x,y\in \obj(\mathcal{G})$ if $\Hom_\mathcal{N}(x,y)\neq \varnothing$ and morphisms $\alpha,\beta\in \mor(\mathcal{G})$ if there exist morphisms $n_1,n_2\in \mor(\mathcal{N})$ such that $\alpha=n_1 \beta n_2$. It is easy to see that the classes of objects of $\mathcal{G}$ form a category whose morphisms are the classes of morphisms of $\mathcal{G}$ with composition induced by the composition in $\mathcal{G}$. This category is clearly a groupoid, called the \emph{quotient groupoid}, and will be denoted by $\mathcal{G}/\mathcal{N}$. The canonical projection $q:\mathcal{G}\to \mathcal{G}/\mathcal{N}$ that maps each object and morphism of $\mathcal{G}$ to its equivalence class is a functor and hence a morphism of groupoids. It is easy to prove that, with the notations above, if $F:\mathcal{G}\to\mathcal{H}$ is a morphism of groupoids such that $F$ sends morphisms of $\mathcal{N}$ to identity maps then there exists a unique groupoid morphism $\overline{F}:\mathcal{G}/\mathcal{N}\to\mathcal{H}$ such that $\overline{F}q=F$.

Recall that an \emph{indiscrete category} is a category $\C$ such that for any two objects $x,y \in \obj(\C)$, the set $\Hom_\C(x,y)$ consists of exactly one morphism. Note that any indiscrete category is a connected groupoid. We say that a category $T$ is a \emph{tree} if $LT$ is an indiscrete category and that a category is a \emph{forest} if it is a disjoint union of trees. If $\C$ is a connected small category, a \emph{maximal tree} in $\C$ is a subcategory $T\subseteq \C$ such that $T$ is a tree and $\obj(T)=\obj(\C)$.

Now, let $\C$ be a connected small category and let $T$ be a maximal tree in $\C$. Clearly $LT$ is a normal subgroupoid of $L\C$ and, since $LT$ is connected, the quotient $L\C/LT$ is actually a group, which is easily seen to be canonically isomorphic to $\aut_{L\C}(c_0)$ for every $c_0\in \obj(\C)$.

Let $\C$ and $\D$ be connected small categories, let $T_\C$ and $T_\D$ be maximal trees in $\C$ and $\D$ respectively and let $F:\C\to\D$ be a functor such that $F(T_\C)\subseteq T_\D$. Given $c_0\in\obj(\C)$, there is a commutative diagram
\begin{center}
\begin{tikzpicture}[x=3cm,y=2.5cm]
	\draw (0,1) node(C){$\C$};
	\draw (1,1) node(LC){$L\C$};
	\draw (2,1) node(LCLT){$L\C/LT_\C$};
	\draw (3,1) node(pi1C){$\pi_1(\C,c_0)$};
	\draw (0,0) node(D){$\D$};
	\draw (1,0) node(LD){$L\D$};
	\draw (2,0) node(LDLT){$L\D/LT_\D$};
	\draw (3,0) node(pi1D){$\pi_1(\D,F(c_0))$};
	\draw[->] (C) -- (LC) node [midway,above] {$\iota_\C$};
	\draw[->] (LC) -- (LCLT) node [midway,above] {$q_\C$};
	\draw[->] (LCLT) -- (pi1C) node [midway,above] {$\cong$};
	\draw[->] (D) -- (LD) node [midway,below] {$\iota_\D$};
	\draw[->] (LD) -- (LDLT) node [midway,below] {$q_\D$};
	\draw[->] (LDLT) -- (pi1D) node [midway,below] {$\cong$};
	\draw[->] (C) -- (D) node [midway,left] {$F$};
	\draw[->] (LC) -- (LD) node [midway,left]{$LF$};
	\draw[->] (LCLT) -- (LDLT) node [midway,right]{$\overline{LF}$};
	\draw[->] (pi1C)-- (pi1D) node [midway,right]{$F_*$};
\end{tikzpicture}
\end{center}
where $q_\C$ and $q_\D$ are the canonical projections.

Now, given a connected small category $\C$, a maximal tree $T$ in $\C$ and a group homomorphism $\alpha:\pi_1(\C)\to G$, we have a functor $\mathcal{F}_{T,\alpha}:\C\to G$ induced by $T$ and $\alpha$ given by the composition $\C\xrightarrow{\iota_\C}L\C\xrightarrow{q} L\C/LT\xrightarrow{\cong} \pi_1(\C)\xrightarrow{\alpha}G$. Observe that $\mathcal{F}_{T,\alpha}$ is trivial on $T$. It is not difficult to prove that the morphism induced by $\mathcal{F}_{T,\alpha}$ in the fundamental groups is $\alpha$.

Note that if $A$ is a connected subcategory of $\C$ we can choose a maximal tree $T_A$ of $A$ and extend it to a maximal tree $T$ of $\C$. Let $i:A\to \C$ be the inclusion functor and let $\alpha:\pi_1(\C)\to G$ be a morphism of groups. From the diagram above, it follows that the functor $\mathcal{F}_{T_A,\alpha i_\ast}:A\to G$ is the restriction of $\mathcal{F}_{T,\alpha}:\C\to G$. In particular, if $i_*=0$, or more generally if $\im i_\ast \subseteq \ker \alpha$, then $\mathcal{F}_{T,\alpha}|_A$ is trivial. If the subcategory $A$ is not connected, the same reasoning applies to each connected component of $A$. \label{functor_and_maximal_tree}

Now, let $X$ be a connected Alexandroff T$_0$--space, let $*$ be the only object of $G$ and let $\mathcal{F}=\mathcal{F}_{T,\alpha}$. An explicit computation of the homotopy fiber $\mathcal{F}/*$  shows that $\obj(\mathcal{F}/*)=X\times G$ and, for $x,y\in X$ there exists a unique morphism in $\mathcal{F}/*$ from $(x,g)$ to $(y,h)$ if an only if $x\leq y$ and $h\mathcal{F}(\phi)=g$, where $\phi$ is the only morphism in $X$ from $x$ to $y$. Therefore $\mathcal{F}/*$ is a poset and if $X$ is a locally finite T$_0$--space then $\mathcal{F}/*$ is the poset associated to a locally finite T$_0$--space. The projection $p:\mathcal{F}/*\to X$ is easily seen to be a covering of spaces so $\mathcal{F}/*$ is a cover of $X$. 

The change of fiber $g_*:\mathcal{F}/*\to \mathcal{F}/*$ induced by $g\in G$ is just the isomorphism $(x,h)\mapsto (x,gh)$ and hence, by Quillen's theorem B, $\pi_1(\mathcal{F}/*)\cong \ker \alpha$ and $\pi_0(\mathcal{F}/*)=G/\im \alpha$. In particular, $\mathcal{F}/*$ is connected if and only if $\alpha$ is an epimorphism, in which case $\mathcal{F}/*$ is the (connected) cover of $X$ corresponding to $\ker\alpha$.

\section{Poset splitting} \label{section_splitting}

In this section we will introduce the \emph{poset splitting} technique and give many interesting results that can be obtained from its application. Moreover, we will develop the necessary results to achieve in the next two sections the main goal of this article: a thorough study of the minimal finite models of the real projective plane, the torus and the Klein bottle which, in its first step, gives answers to two open questions.

The splitting technique consists in dividing a poset $X$ into two (not necessarily connected) subspaces, say $C$ and its complement, and from homotopical (or homological) information of $C$ and $X-C$ obtain homotopical or homological information of $X$.
This seems odd at first sight since one would expect a different situation, such as covering the space $X$ with two open subspaces. However, applying lemma \ref{lemma_subdiv_X_A} we will be able to extend arguments for open covers $\{U,V\}$ to any cover $\{C,D\}$ of $X$.

Before getting into the details of this extension, note that if $X$ is a poset and $U$ and $V$ are open subspaces of $X$ such that $X=U\cup V$ then $X$ is the non-Hausdorff homotopy pushout \cite[Definition 2.1]{FM} of the diagram $U\hookleftarrow U\cap V \hookrightarrow V$ and thus, by Thomason's theorem \cite{Tho}, $|\K(X)|$ is the homotopy pushout of the diagram $|\K(U)|\hookleftarrow |\K(U\cap V)| \hookrightarrow |\K(V)|$ (cf. \cite{FM}). Note also that $(|\K(X)|;|\K(U)|,|\K(V)|)$ is a CW-triad. Hence, several tools may be applied to obtain homotopical information of $X$ from homotopical information of $U$ and $V$. And considering opposite spaces one can apply the previous argument also if $U$ and $V$ are closed subspaces of $X$.

However, the requirement that $U$ and $V$ be open (or closed) subspaces of $X$ is too restrictive to tackle some combinatorial problems such as the minimality of finite models. To get rid of this hypothesis we may apply subdivision of posets as follows.

Let $C$ and $D$ be subspaces of $X$ such that $X=C\cup D$. Then $C'$ and $D'$ are open subspaces of $X'$ and $X'=(X'-C')\cup(X'-D')$. By the argument above, we might be able to obtain homotopical information of $X$ from $X'-C'$ and $X'-D'$, which by lemma \ref{lemma_subdiv_X_A}, are weak homotopy equivalent to $(X-C)'$ and $(X-D)'$ respectively. Hence, if $C$ and $D$ are disjoint subspaces of $X$ such that $X=C\cup D$ one may expect to obtain homotopical information of $X$ from $C'$ and $D'$, and hence from $C$ and $D$.

As a first simple example consider the following result.

\begin{prop}
Let $X$ be an Alexandroff T$_0$--space. Suppose that there exist subspaces $C,D\subseteq X$ such that $X=C\cup D$ and such that the morphisms $H_1(C)\rightarrow H_1(X)$ and $H_1(D)\rightarrow H_1(X)$ induced by the inclusion maps are trivial. Then $H_1(X)$ is a free abelian group.
\end{prop}

\begin{proof}
We may assume that $C,D\subsetneq X$. Observe that the inclusion $(X-C)'\to X'$ induces the trivial map in $H_1$ since $(X-C)'\subseteq D'$. Now, by \ref{lemma_subdiv_X_A}, it follows that the inclusion $X'-C'\to X'$ also induces the trivial map in $H_1$. Similarly, the inclusion $X'-D'\to X'$ induces the trivial map in $H_1$.

Since $X'-C'$ and $X'-D'$ are closed subspaces of $X'$ and $X'=(X'-C')\cup(X'-D')$, considering the opposite spaces we obtain the following portion of the Mayer-Vietoris exact sequence
$$H_1(X'-C')\oplus H_1(X'-D') \xrightarrow{0} H_1(X')\xrightarrow{\partial} H_0((X'-C')\cap (X'-D'))$$
And since $H_0((X'-C')\cap (X'-D'))$ is a free abelian group and $\partial$ is a monomorphism we obtain that $H_1(X')\cong H_1(X)$ is a free abelian group.
\end{proof}

A more sophisticated example of application of poset splitting is given by the following theorem about the homology groups of a cover of a poset which is obtained combining the splitting idea with theorem 5.1 of \cite{CO} and with the arguments of coverings of posets given in the second part of section 2.

\begin{theo} \label{theo_homology_cover_space}
Let $X$ be a connected Alexandroff T$_0$--space. Let $C\subsetneq X$ be a subspace and let $i:X-C\to X$ be the inclusion map. Let $G$ be a group and let $\alpha:\pi_1(X)\to G$ be an epimorphism such that $i_\ast(\pi_1(X-C,x_0))\subseteq \ker \alpha$ for all $x_0\in X-C$. Let $p:\widetilde X\to X$ be a covering map (with $\widetilde X$ connected) such that $p_\ast(\pi_1 (\widetilde X))=\ker \alpha$.

Then,
$$H_n(\widetilde X, p^{-1}(C))=\bigoplus_{G}H_n(X,C)$$
for all $n\in \N_0$.
\end{theo}

\begin{proof}
By lemma \ref{lemma_subdiv_X_A}, taking the barycentric subdivisions of the posets $X$ and $C$ we may suppose that $C$ is an open subspace of $X$ and that $X$ is a quasicellular poset and hence that $(X,C)$ is a relative quasicellular pair. 

By the arguments in the previous section, there exists a functor $\mathcal{F}:X\to G$ such that $\mathcal{F}|_{X-C}$ is trivial and such that the homotopy fiber $\mathcal{F}/\ast$ is a locally finite T$_0$--space which is the covering space of $X$ corresponding to the subgroup $\ker \alpha$ of $\pi_1(X)$. We may assume that $\widetilde X=\mathcal{F}/\ast$.

Let $\rho$ be a quasicellular morphism for $(X,C)$. Note that $(\widetilde X,p^{-1}(C))$ is a relative quasicellular pair with quasicellular morphism induced by $\rho$. 

Since $\mathcal{F}|_{X-C}$ is trivial we obtain that 
$$\left(C_n(\widetilde X,p^{-1}(C)),\widetilde d_n\right)_{n\in\Z}=\left( \bigoplus_{G}C_n(X,C), \bigoplus_{G}d_n \right)_{n\in\Z}$$
where $d_n$ and $\widetilde d_n$ are the differentials defined in \ref{coro_quasicel_relativo} for the relative quasicellular pairs $(X,C)$ and $(\widetilde X,p^{-1}(C))$ respectively. The result follows.
\end{proof}

From this theorem we obtain the following corollary which will be applied later.

\begin{coro} \label{coro_homology_cover_space}
Let $X$ be a connected Alexandroff T$_0$--space. Let $C\subsetneq X$ be a subspace such that the inclusion of each connected component of $X-C$ in $X$ induces the trivial morphism between the fundamental groups. Let $p:\widetilde X\to X$ be the universal cover of $X$. 

Then
$$H_n(\widetilde X, p^{-1}(C))\cong\bigoplus_{\pi_1(X)}H_n(X,C)$$
for all $n\in \N_0$.
\end{coro}

This corollary motivates the following definition.

\begin{definition}
Let $X$ be a connected Alexandroff T$_0$--space and let $C$ and $D$ be non-empty subspaces of $X$ such that $X=C\cup D$. We say that the 3-tuple $(X;C,D)$ is a \emph{splitting triad} if the inclusion of each connected component of $D$ in $X$ induces the trivial morphism between the fundamental groups.
\end{definition}

Observe that a splitting triad $(X;C,D)$ gives rise to a triad $(X;C,X-C)$ that \emph{splits} $X$ into two disjoint subspaces satisfying the hypothesis of the previous corollary. In this case, $(X;C,X-C)$ is also a splitting triad. 

We will be specially interested in splitting triads $(X;C,X-C)$ such that the subspace $C$ satisfies additional hypotheses. Under adequate conditions we will be able to obtain concrete homotopical information of $X$ as the following Hurewicz-type theorem shows.

\begin{theo} \label{theo_pi_2_v2}
Let $n\in\N$ with $n\geq 2$ and let $(X;C,D)$ be a splitting triad such that
\begin{itemize}
\item $\pi_r(C,c_0)\cong H_r(C)=0$ for all $2\leq r\leq n$ and for all $c_0\in C$
\item $\pi_1(C,c_0)$ is an abelian group for all $c_0\in C$
\item the inclusion map $C\to X$ induces a monomorphism $H_1(C)\to H_1(X)$
\end{itemize}
If $l\in\N$ is such that $2\leq l\leq n$ and $H_j(X)=0$ for all $2\leq j\leq l-1$ then $\pi_j(X)=0$ for all $2\leq j\leq l-1$ and 
$\pi_l(X)\cong H_l(X)\otimes\Z[\pi_1(X)]$. In particular, $\pi_2(X) \cong H_2(X)\otimes\Z[\pi_1(X)]$.
\end{theo}

\begin{proof}
Let $p:\widetilde{X}\to X$ be the universal cover of $X$.

Let $i:C\to X$ be the inclusion map. Since $\pi_1(C,c_0)$ is an abelian group for all $c_0\in C$ and $i_\ast:H_1(C)\to H_1(X)$ is a monomorphism we obtain that $i_\ast:\pi_1(C,c_0)\to \pi_1(X,c_0)$ is a monomorphism for all $c_0\in C$. Hence, for every connected component $C'$ of $C$ and  every connected component $D$ of $p^{-1}(C')$ we obtain that $p|_D:D\to C'$ is a universal cover and thus $D$ is $n$-connected since $\pi_r(C,c_0)=0$ for all $2\leq r\leq n$ and for all $c_0\in C$.

Therefore, from the long exact sequence in homology for the topological pair $(\widetilde{X},p^{-1}(C))$ we obtain that $H_r(\widetilde{X})\cong H_r(\widetilde{X},p^{-1}(C))$ for all $2\leq r\leq n$.

On the other hand, since the inclusion of each connected component of $X-C$ in $X$ induces the trivial morphism between the fundamental groups, from \ref{theo_homology_cover_space} we obtain that $\displaystyle H_r(\widetilde{X},p^{-1}(C))\cong \bigoplus_{\pi_1(X)} H_r(X,C)$ for all $r\in\N_0$.

Finally, $H_r(X,C)\cong H_r(X)$ for all $2\leq r\leq n$ since $H_r(C)=0$ for all $2\leq r\leq n$ and $i_\ast:H_1(C)\to H_1(X)$ is a monomorphism. The result follows.
\end{proof}

Note that the hypotheses of the previous theorem are fulfilled if all the connected components of $C$ are $2$--connected.

The following is a simple example of application of \ref{theo_pi_2_v2}.

\begin{ex}
Consider the poset $Z$ defined by the following Hasse diagram
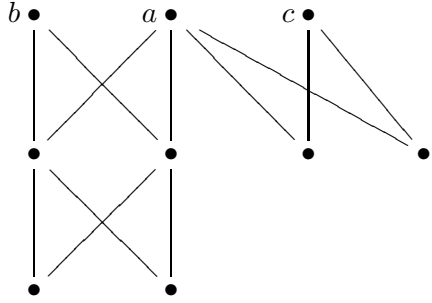
\[
\xymatrix@R=40pt{
b\bullet\phantom{b}\ar@{-}[d]\ar@{-}[dr] & a\bullet\phantom{a}\ar@{-}[dl]\ar@{-}[d]\ar@{-}[dr]\ar@{-}[drr] & c\bullet\phantom{c}\ar@{-}[d]\ar@{-}[dr] \\
\bullet\ar@{-}[d]\ar@{-}[dr] & \bullet \ar@{-}[dl]\ar@{-}[d] & \bullet & \bullet \\
\bullet & \bullet \\
}
\]
We wish to compute $\pi_2(Z)$.

Clearly, the inclusion of each connected component of $Z-U_a$ in $Z$ induces the trivial morphism between the fundamental groups, and since $U_a$ is contractible, the previous theorem applies and yields
\[\pi_2(Z)=H_2(Z)\otimes\Z[\pi_1(Z)].\]

Now, it is easy to verify that $H_2(Z)\cong\Z$ (cf. example 5.4 of \cite{CO}).

On the other hand, the fundamental group of $Z$ is also easy to compute. Considering the open sets $U_a\cup U_c$ and $U_b$ and applying the van Kampen theorem we obtain that
\[\pi_1(Z) \cong \pi_1(U_a\cup U_c) \cong \pi_1(U_c\cup\{a\}) \cong \Z\ .\]
Here, the first isomorphism holds since the map $\pi_1(U_b\cap (U_a\cup U_c))\to \pi_1(U_a\cup U_c)$ induced by the inclusion is trivial as it can be factorized through $\pi_1(U_a)$. And the second isomorphism holds since $U_c\cup\{a\}$ can be obtained from $U_a\cup U_c$ by removing beat points.

Thus, $\pi_2(Z)=\Z[\Z]$ (indeed, $\mathcal{K}(Z)$ is homeomorphic to $S^2\vee S^1$).
\end{ex}

The following theorem is another example of application of splitting triads. As a corollary we will obtain a generalization of a result of R. Brown.

\begin{theo} \label{theo_splitting_free_fundamental_group}
Let $X$ be a connected Alexandroff T$_0$--space. Suppose that there exist nonempty subspaces $C,D\subseteq X$ such that $X=C\cup D$ and such that the inclusions of each connected component of $C$ and $D$ in $X$ induce the trivial morphism between the fundamental groups. Then $\pi_1(X)$ is a free group.
\end{theo}

\begin{proof}
We may assume that $C,D\subsetneq X$. Since $X-C\subseteq D$ then the inclusion of each connected component of $X-C$ in $X$ induces the trivial morphism between the fundamental groups. Thus, by lemma \ref{lemma_subdiv_X_A}, the inclusion of each component of $C'$ and $X'-C'$ in $X'$ induces the trivial morphism between the fundamental groups. It follows that, without loss of generality, we can assume that $X$ is a locally finite T$_0$--space and that $C$ is an open subspace of $X$.

Choosing a maximal tree on every connected component of $C$ and $X-C$ we obtain a forest $T$ which can be extended to a maximal tree $\widetilde{T}$ of $X$. 
Consider the functor $\mathcal{F}=\mathcal{F}_{\widetilde{T},\id}:X\to \pi_1(X)$ associated to $\widetilde{T}$ and the identity morphism $\id:\pi_1(X)\to\pi_1(X)$. Since the inclusion of each connected component of $C$ and $X-C$ induces the trivial morphism between fundamental groups, by the argument in page \pageref{functor_and_maximal_tree} we obtain that $\mathcal{F}$ is trivial on $C$ and $X-C$.

We define an equivalence relation on $X$ as follows. Two objects $x,y\in \obj(X)$ are equivalent if and only if they both lie in the same connected component of $C$ or $X-C$ and two arrows $\alpha,\beta\in \mor(X)$ are equivalent if and only if their domains and codomains are respectively equivalent and $\mathcal{F}(\alpha)=\mathcal{F}(\beta)$. For $x\in \obj(X)$ and $\alpha\in \mor(X)$ we will denote the equivalence class of $x$ by $[x]$ and the equivalence class of $\alpha$ by $[\alpha]$.

Let $\CC$ be the category whose objects are the equivalence classes of objects of $X$ and whose morphisms are the equivalence classes of arrows of $X$, that is, $\Hom_{\CC}([x],[y])=\{[\alpha]:\alpha\in \Hom_X(x',y'), x\in [x], y'\in [y]\}$ for $[x],[y] \in \obj(\CC)$. Observe that since $\mathcal{F}$ is trivial on every connected component of $C$ and $X-C$, then the domain and codomain of every non-identity morphism in $\CC$ are the class of an element of $C$ and the class of an element of $X-C$ respectively.
In particular, $\CC$ does not have non-trivial compositions.  It follows that every non-degenerate simplex of the (simplicial) nerve of $\CC$ has dimension at most $1$ and hence the classifying space of $\CC$ is homotopically equivalent to a wedge sum of circumferences. Consequently the fundamental group of $\CC$ is a free group.

Let $\phi:X\to \CC$ be the function that maps every object and every morphism of $X$ to its class. It is easy to check that $\phi$ is a well defined functor from $X$ to $\CC$.

It is clear that $\mathcal{F}$ factors through $\mathcal{C}$ as $\mathcal{F}=\widetilde{\mathcal{F}} \phi$. Since $\mathcal{F}$ induces the identity on $\pi_1(X)$ we obtain that $\phi_*:\pi_1(X)\to \pi_1(\mathcal{C})$ is a monomorphism. And since $\pi_1(\mathcal{C})$ is a free group, it follows that $\pi_1(X)$ is also a free group as desired.
\end{proof}

It is not hard to prove that the categories $L\CC$ and $LX/LT$ of the previous proof are isomorphic groupoids, and that $\pi_1(X)\cong \pi_1(\CC)$.

It is important to observe that results \ref{theo_homology_cover_space}, \ref{coro_homology_cover_space}, \ref{theo_pi_2_v2} and \ref{theo_splitting_free_fundamental_group} also hold for arbitrary topological spaces replacing the corresponding splitting triads by excisive triads $(X;X_1,X_2)$ such that the inclusion $X_2 \hookrightarrow X$ induces the trivial map between the fundamental groups for any choice of the base point of $X_2$. Indeed, such an excisive triad may be replaced by a weak equivalent CW-triad $(Z;A,B)$ such that $Z$ is a regular CW-complex, to obtain then a splitting triad $(\X(Z);\X(A),\X(B))$.

In particular, \ref{theo_splitting_free_fundamental_group} yields the following generalization of corollary 3.7 of \cite{Bro}.

\begin{coro}
Let $X$ be a path-connected topological space, let $x_0\in X$ and let $X_1$ and $X_2$ be subspaces of $X$ such that their interiors cover $X$ and such that for $j\in\{1,2\}$ the inclusion of each path-connected component of $X_j$ in $X$ induces the trivial morphism between the fundamental groups. Then, $\pi_1(X,x_0)$ is a free group.
\end{coro}

And from \ref{theo_pi_2_v2} we obtain the following generalization of Hurewicz's theorem.

\begin{theo} \label{theo_pi_2_v2_spaces}
Let $X$ be a topological space and let $n\in\N$ with $n\geq 2$. Suppose that there exist non-empty subspaces $C,D\subseteq X$ such that their interiors cover $X$ and such that they satisfy the following:
\begin{itemize}
\item The subspace $C$ is $n$--connected.
\item The inclusion of each connected component of $D$ into $X$ induces the trivial map between the fundamental groups.
\end{itemize}
If $l\in\N$ is such that $2\leq l\leq n$ and $H_j(X)=0$ for all $2\leq j\leq l-1$ then $\pi_j(X)=0$ for all $2\leq j\leq l-1$ and 
$\pi_l(X)\cong H_l(X)\otimes\Z[\pi_1(X)]$. In particular, $\pi_2(X) \cong H_2(X)\otimes\Z[\pi_1(X)]$.
\end{theo}

\subsection*{Splitting properties}

We will now define certain properties on splitting triads of finite and connected T$_0$--spaces which will be called \emph{splitting properties}. This properties are essential for the results of the next two sections since they will provide the link between Topology and Combinatorics that is needed to solve the problems on minimality of finite models we address in this article.

\begin{definition} \label{def_S1_S2}
Let $X$ be a finite and connected T$_0$--space and let $(X;C,D)$ be a splitting triad. We define the following properties on $(X;C,D)$:
\begin{itemize}
\item[{\Spl[1]}] The inclusion of each connected component of $C$ in $X$ induces the trivial morphism between the fundamental groups, that is, the triad $(X;D,C)$ is also a splitting triad.
\item[{\Spl[2]}] At most one of the connected components of $C$ is weak homotopy equivalent to $S^1$ and the rest of them are homotopically trivial.
\end{itemize}
We will also say that a finite and connected T$_0$--space $X$ satisfies \Spl[1] (resp. \Spl[2]) if there exists a splitting triad  $(X;C,D)$ that satisfies \Spl[1] (resp. \Spl[2]).
\end{definition}

Note that, by \ref{theo_splitting_free_fundamental_group}, if $X$ satisfies \Spl[1] then $\pi_1(X)$ is a free group.

The following remark sums up easy facts about conditions \Spl[1] and \Spl[2]. Recall that if $X$ is a finite poset, the \emph{height} of $X$, denoted $h(X)$, is one less than the maximum cardinality of a chain of $X$. Note that $h(X)=\dim (\K(X))$.

\begin{rem} \label{rem_mxl_and_cond} 
Let $X$ be a finite and connected T$_0$--space.
\begin{enumerate}
\item If $X$ has a maximum, then $X$ satisfies \Spl[1] and \Spl[2]. 
\item If $h(X)= 1$, then $X$ satisfies \Spl[1] and \Spl[2]. Indeed, in this case if $C$ is the set of maximal points of $X$ then the splitting triad $(X;C,X-C)$ satisfies \Spl[1] and \Spl[2].
\item If $X$ is the union of two contractible subspaces then $X$ satisfies \Spl[1] and \Spl[2]. In particular, if $X$ has at most $2$ maximal points (or at most $2$ minimal points) then $X$ satisfies \Spl[1] and \Spl[2].
\end{enumerate}
\end{rem}

\begin{ex}
If $X$ is a finite model of the real projective plane then $X$ does not satisfy \Spl[1] since its fundamental group is not free. In particular, the space $\Pp^2_2$ of figure \ref{fig_P2_2} of page \pageref{fig_P2_2} does not satisfy \Spl[1]. However, it satisfies \Spl[2] since the splitting triad $(\Pp^2_2;\Pp^2_2-U_{a_1},U_{a_1})$ satisfies \Spl[2].
\end{ex}

\begin{prop} \label{prop_S123_beat_points}
Let $X$ be a finite and connected T$_0$--space and let $a$ be a beat point of $X$. If $X-\{a\}$ satisfies \Spl[1] \textnormal{(}resp. \Spl[2]\textnormal{)} then $X$ satisfies \Spl[1] \textnormal{(}resp. \Spl[2]\textnormal{)}.
\end{prop}

\begin{proof}
Follows from the fact that if $a$ is a beat point of $X$ and $C\subseteq X-\{a\}$ then either $a$ is a beat point of $C\cup \{a\}$ or $a$ is a beat point of $X-C$.
\end{proof}

\begin{prop} \label{prop_S2_KG1}
Let $X$ be a finite T$_0$--space such that $X$ is an Eilenberg-MacLane space of type $(G,1)$. If $X$ satisfies \Spl[2] then $H_1(X)$ is free abelian and $H_n(X)=0$ for all $n\geq 2$.
\end{prop}

\begin{proof}
Note that the group $G$ has finite cohomological dimension over $\Z$ since $|\K(X)|$ is a finite aspherical CW-complex with fundamental group isomorphic to $G$. Hence, $G$ is torsion-free.

Let $(X;C,D)$ be a splitting triad that satisfies \Spl[2] and let $i:C\to X$ be the inclusion map. If the inclusion of each connected component of $C$ in $X$ induces the trivial map between the fundamental groups then, from \ref{theo_splitting_free_fundamental_group} we obtain that $\pi_1(X)$ is a free group. Thus, $X$ is homotopy equivalent to a wedge of circumferences and hence $H_1(X)$ is free abelian and $H_n(X)=0$ for all $n\geq 2$.

Otherwise, for all $c_0\in C$, the induced map $i_\ast:\pi_1(C,c_0)\to\pi_1(X,c_0)$ is a monomorphism since $\pi_1(X)\cong G$ is torsion-free.

Let $p:\widetilde{X}\to X$ be the universal cover of $X$.

Since $i_\ast:\pi_1(C,c_0)\to \pi_1(X,c_0)$ is a monomorphism for all $c_0\in C$, for every connected component $C'$ of $C$ and  every connected component $D$ of $p^{-1}(C')$ we obtain that $p|_D:D\to C'$ is a universal cover and thus $D$ homotopically trivial since $\pi_r(C,c_0)=0$ for all $r\geq 2$ and for all $c_0\in C$.

Therefore, from the long exact sequence in homology for the topological pair $(\widetilde{X},p^{-1}(C))$ we obtain that $H_r(\widetilde{X},p^{-1}(C))=0$ for all $r\neq 1$ and $H_1(\widetilde{X},p^{-1}(C))\cong \widetilde H_0(p^{-1}(C))$ which is free abelian.

On the other hand, since the inclusion of each connected component of $X-C$ in $X$ induces the trivial morphism between the fundamental groups, from \ref{coro_homology_cover_space} we obtain that $\displaystyle H_r(\widetilde{X},p^{-1}(C))\cong \bigoplus_{\pi_1(X)} H_r(X,C)$ for all $r\in\N_0$.

Hence, $H_r(X,C)=0$ for all $r\neq 1$ and $H_1(X,C)$ is free abelian. Thus, $H_1(X)$ is free abelian and $H_n(X)=0$ for all $n\geq 2$.
\end{proof}

From the previous proposition we obtain the following corollary which will be used in section \ref{section_torus} to characterize the minimal finite models of the torus and the Klein bottle.

\begin{coro} \label{coro_torus_Klein_bottle_not_S2}
Let $X$ be a finite T$_0$--space. If $X$ is a finite model of either the torus or the Klein bottle then $X$ does not satisfy \Spl[2].
\end{coro}

Condition \Spl[2] permits to give yet another Hurewicz-type theorem which will be applied in example \ref{ex_S1_and_not_S2}.

\begin{prop} \label{prop_S2_pi3}
Let $X$ be a finite and connected T$_0$--space such that $\pi_2(X)=0$ and such that $X$ satisfies \Spl[2]. If $n\geq 3$ is such that $H_l(X)=0$ for all $3\leq l \leq n-1$ then $\pi_l(X)=0$ for all $3\leq l \leq n-1$ and $\pi_{n}(X)=H_{n}(X)\otimes \Z[\pi_1(X)]$.
In particular, $\pi_3(X)=H_3(X)\otimes \Z[\pi_1(X)]$.
\end{prop}

\begin{proof}
Let $(X;C,X-C)$ be a splitting triad that satisfies \Spl[2] and let $p:\widetilde{X}\to X$ be the universal cover of $X$. For every connected component $C'$ of $C$ and every connected component $D$ of $p^{-1}(C')$ we obtain that $p|_D:D\to C'$ is a covering map and thus $D$ is either homotopically trivial or weak homotopy equivalent to $S^1$. In any case, $H_r(D)=0$ for all $r\geq 2$. Therefore, from the long exact sequence in homology for the topological pair $(\widetilde{X},p^{-1}(C))$ we obtain that $H_r(\widetilde{X})\cong H_r(\widetilde{X},p^{-1}(C))$ for all $r\geq 3$.

On the other hand, since the inclusion of each connected component of $X-C$ in $X$ induces the trivial morphism between the fundamental groups, from \ref{theo_homology_cover_space} we obtain that $\displaystyle H_r(\widetilde{X},p^{-1}(C))\cong \bigoplus_{\pi_1(X)} H_r(X,C)$ for all $r\in\N_0$.

Thus, $\displaystyle H_r(\widetilde{X})\cong \bigoplus_{\pi_1(X)} H_r(X,C) \cong \bigoplus_{\pi_1(X)} H_r(X)$ for all $r\geq 3$. The result follows.
\end{proof}

The following example shows a a finite and connected T$_0$--space which satisfies \Spl[1] and does not satisfy \Spl[2]. Recall that $S^0$ denotes the $0$--sphere.

\begin{ex} \label{ex_S1_and_not_S2}
Let $X=\Su S^0 \times \Su^3 S^0$. Note that $X$ is a finite model of $S^1\times S^3$. It follows that $\pi_1(X)\cong \Z$, $\pi_2(X)=0$ and $\pi_3(X)\cong \Z$. Thus, from \ref{prop_S2_pi3} we obtain that $X$ does not satisfy \Spl[2].

Now, let $w$ be a minimal element of $\Su S^0$ and let $C=\{w\}\times \Su^3 S^0$. It is easy to check that $C$ and $X-C$ are both homotopy equivalent to $\Su^3 S^0$. Thus $X$ satisfies \Spl[1].
\end{ex}

\section{Setting up} \label{section_setting_up}

In this short section we will introduce the specific notation that will be used in the next two sections. We will also define two relations which will be applied to study the problems of minimality of finite models.

If $X$ is a poset, $\mxl(X)$ and $\mnl(X)$ will denote the subsets of maximal and minimal points of $X$ respectively. Also, if $a\in\mxl(X)$, we define $V_a^X=U_a^X\cup \mnl(X)$ and $W_a^X=U_a^X\cup \mnl(X) \cup \mxl(X)$.

\begin{rem} \label{rem_minimal_space_two_points}
Let $X$ be a finite T$_0$--space without beat points. If $a\in X-\mxl(X)$ then $\#(\widehat F_a \cap \mxl(X) )\geq 2$. Similarly, if $b\in X-\mnl(X)$ then $\#(\widehat U_b \cap \mnl(X) )\geq 2$.
\end{rem}

\begin{definition}
Let $X$ be a finite and connected T$_0$--space. We define
\begin{itemize}
\item $\M=\mxl(X)\times \mnl(X)$
\item $\B=X-(\mxl(X)\cup\mnl(X))$
\item $\alpha_b=\# (F_b\cap\mxl(X))$, for each $b\in\B$
\item $\beta_b=\# (U_b\cap\mnl(X))$, for each $b\in\B$
\item $m_X=\#\mxl(X)$
\item $n_X=\#\mnl(X)$
\item $l_X=\#\B$
\end{itemize}
We also define the relation $\Relm\subseteq \B\times\mxl(X)$ by
$$b\,\Relm\,a \Leftrightarrow b\notin U_a$$
and the relation $\Relp\subseteq \B\times\M$ by
$$b\,\Relp\, (x,y) \Leftrightarrow b\notin U_x \textnormal{ and } b\notin F_y$$
We shall frequently write $m$, $n$ and $l$ in place of $m_X$, $n_X$ and $l_X$ when there is no risk of confusion.
\end{definition}

Note that $\#\Relm(b)=m_X-\alpha_b$ and $\#\Relp(b)=(m_X-\alpha_b)(n_X-\beta_b)$ for all $b\in \B$. Note also that, if $X$ does not have beat points then $\alpha_b\geq 2$ and $\beta_b\geq 2$ for all $b\in \B$ by \ref{rem_minimal_space_two_points}.

Finally, if $m\in \N$, $D_m$ will denote the discrete space of cardinality $m$.

\section{Minimal finite models of the real projective plane} \label{section_real_projective_plane}

In this section we will show that the poset of figure \ref{fig_finite_P2_1} below is a minimal finite model of the real projective plane, answering an open question of Barmak \cite[p. 94]{BarLN}. Indeed, we will prove a stronger result which states that if $X$ is a poset with fewer than 13 points then $X$ satisfies \Spl[1] (see definition \ref{def_S1_S2} above) and thus, by \ref{theo_splitting_free_fundamental_group}, $\pi_1(X)$ is a free group. In addition, we will prove that if $X$ is a poset with fewer than 13 points then $H_n(X)$ is torsion-free for all $n\in\N$ settling a conjecture of Hardie, Vermeulen and Witbooi \cite[Remark 4.2]{HVW}. 

Furthermore, we will show that the poset of figure \ref{fig_finite_P2_1} below and its opposite are the only minimal finite models of the real projective plane.

Recall that the finite model of the real projective plane of figure \ref{fig_finite_P2_1} is constructed in \cite[Example 7.1.1]{BarLN} as the poset of cells of the regular CW-complex structure of figure \ref{fig_CW_P2_1} and was also obtained in \cite{HVW} from a finite version of the mapping cone. This poset will be denoted by $\Pp^{2}_1$ and the real projective plane will be denoted by $\Pp^2$.

  \begin{figure}[h!]
  \begin{subfigure}[t]{0.45\textwidth}
  \begin{center}
\begin{tikzpicture}[x=4cm,y=4cm]
   \tikzstyle{every node}=[font=\normalsize]
	\draw
		(0,0) node[below left=-1]{$c_1$}--
		(1,0) node[below right=-1]{$c_2$}--
		(1,1) node[above right=-1]{$c_1$}--
		(0,1) node[above left=-1]{$c_2$}--
		(0,0);
	\draw
		(0,0)--
		(0.5,0.5) node[left]{$c_3$}--
		(0,1);
	\draw (0.5,0.5)--(1,0);
	\draw (0.5,0.5)--(1,1);

	\draw[\arr]  (0,0)--(0.53,0);
	\draw[\arr]  (1,0)--(1,0.53);	
	\draw[\arr]  (1,1)--(0.47,1);
	\draw[\arr]  (0,1)--(0,0.47);

\draw (0.5,0) node[below]{$b_1$};
\draw (1,0.5) node[right]{$b_2$};
\draw (0.5,1) node[above]{$b_1$};
\draw (0,0.5) node[left]{$b_2$};

\draw (0.25,0.25) node[above=-1]{$b_3$};
\draw (0.75,0.75) node[below=-1]{$b_4$};
\draw (0.25,0.75) node[right=-1]{$b_5$};
\draw (0.75,0.25) node[left=-1]{$b_6$};

\draw (0.5,0.25) node[below]{$a_1$};
\draw (0.5,0.75) node[above]{$a_2$};
\draw (0.25,0.5) node[left]{$a_3$};
\draw (0.75,0.5) node[right]{$a_4$};

  \fill[fill=black] (0,0) circle (.2em);
  \fill[fill=black] (0,1) circle (.2em);
  \fill[fill=black] (1,0) circle (.2em);
  \fill[fill=black] (1,1) circle (.2em);
  \fill[fill=black] (0.5,0.5) circle (.2em);
  \end{tikzpicture}
    \caption{Regular CW-complex structure of $\Pp^{2}$.}
    \label{fig_CW_P2_1}
      \end{center}
  \end{subfigure}
  \quad
  \begin{subfigure}[t]{0.45\textwidth}
\begin{tikzpicture}[x=4cm,y=4cm]
	\tikzstyle{every node}=[font=\footnotesize]
	
	\foreach \x in {1,...,4} \draw (0.4*\x,1) node(a\x){$\bullet$} node[above=1]{$a_{\x}$};
	\foreach \x in {1,...,6} \draw (0.28*\x,0.5) node(b\x){$\bullet$} node[right=1]{$b_{\x}$};
	\foreach \x in {1,...,3} \draw (0.5*\x,0) node(c\x){$\bullet$} node[below=1]{$c_{\x}$};
	
	\foreach \x in {1,3,6} \draw (a1)--(b\x);
	\foreach \x in {1,4,5} \draw (a2)--(b\x);
	\foreach \x in {2,3,5} \draw (a3)--(b\x);
	\foreach \x in {2,4,6}  \draw (a4)--(b\x);
	
	\foreach \x in {1,2,3,4} \draw (c1)--(b\x);
	\foreach \x in {1,2,5,6} \draw (c2)--(b\x);
	\foreach \x in {3,4,5,6} \draw (c3)--(b\x);
	 
\end{tikzpicture}
  \caption{Space $\Pp^{2}_1$, finite model of $\Pp^{2}$.}
  \label{fig_finite_P2_1}
  \end{subfigure}
  \caption{Cellular and finite models of $\Pp^{2}$.}
  \label{fig_P2_1}
\end{figure}

Clearly, the space $(\Pp^2_1)^\op$ is another finite model of $\Pp^2$ and will be called $\Pp^2_2$ (see figure \ref{fig_finite_P2_2}). It can also be obtained as the face poset of the regular CW-complex structure of $\Pp^2$ given in figure \ref{fig_CW_P2_2}.

  \begin{figure}[h!]

  \begin{subfigure}[t]{0.45\textwidth}
  \begin{center}
  \begin{tikzpicture}[x=4cm,y=4cm]
	\tikzset{->-/.style={decoration={
	markings,
	mark=at position 0.57 with {\arrow{triangle 45}}},postaction={decorate}}}
	\tikzstyle{every node}=[font=\normalsize]
	\draw [->-] (0.5,0) node[below=-1]{$c_2$}--(0.93,0.25);
	\draw [->-] (0.93,0.25) node[right=-1]{$c_1$}--(0.93,0.75);
	\draw [->-] (0.93,0.75) node[right=-1]{$c_3$}--	(0.5,1);
	\draw [->-] (0.5,1) node[above=-1]{$c_2$}--(0.07,0.75);
	\draw [->-] (0.07,0.75) node[left=-1]{$c_1$}--(0.07,0.25);
	\draw [->-] (0.07,0.25) node[left=-1]{$c_3$}--(0.5,0);
	\draw (0.5,0.5) node[below left=-1]{$c_4$};
	\draw (0.71,0.125) node[below right=-1]{$b_1$};
	\draw (0.93,0.5) node[right]{$b_3$};
	\draw (0.71,0.625) node[below=-1]{$b_2$};
	\draw (0.71,0.875) node[above=-1]{$b_5$};
	\draw (0.29,0.125) node[below=-2]{$b_5$};
	\draw (0.29,0.625) node[below=-2]{$b_6$};
	\draw (0.29,0.875) node[above left=-2]{$b_1$};
	\draw (0.5,0.25) node[left=-3]{$b_4$};
	\draw (0.07,0.5) node[left=-2]{$b_3$};
	\draw (0.29,0.375) node{$a_3$};
	\draw (0.71,0.375) node{$a_1$};
	\draw (0.5,0.75) node{$a_2$};
	\draw (0.5,0.5)--(0.5,0);
	\draw (0.5,0.5)--(0.07,0.75);
	\draw (0.5,0.5)--(0.93,0.75);
	\fill[fill=black] (0.5,0.5) circle (.2em);
	\fill[fill=black] (0.5,0) circle (.2em);
	\fill[fill=black] (0.5,1) circle (.2em);
	\fill[fill=black] (0.07,0.75) circle (.2em);
	\fill[fill=black] (0.07,0.25) circle (.2em);
	\fill[fill=black] (0.93,0.75) circle (.2em);
	\fill[fill=black] (0.93,0.25) circle (.2em);
\end{tikzpicture}
    \caption{Regular CW-complex structure of $\Pp^{2}$.}
    \label{fig_CW_P2_2}
      \end{center}
  \end{subfigure}
  \quad
  \begin{subfigure}[t]{0.45\textwidth}
\begin{tikzpicture}[x=4cm,y=4cm]
	\tikzstyle{every node}=[font=\footnotesize]
	
	\foreach \x in {1,...,3} \draw (0.5*\x,1) node(a\x){$\bullet$} node[above=1]{$a_{\x}$};
	\foreach \x in {1,...,6} \draw (0.28*\x,0.5) node(b\x){$\bullet$} node[right=1]{$b_{\x}$};
	\foreach \x in {1,...,4} \draw (0.4*\x,0) node(c\x){$\bullet$} node[below=1]{$c_{\x}$};
	
	\foreach \x in {1,2,3,4} \draw (a1)--(b\x);
	\foreach \x in {1,2,5,6} \draw (a2)--(b\x);
	\foreach \x in {3,4,5,6} \draw (a3)--(b\x);
	
	\foreach \x in {1,3,6} \draw (c1)--(b\x);
	\foreach \x in {1,4,5} \draw (c2)--(b\x);
	\foreach \x in {2,3,5} \draw (c3)--(b\x);
	\foreach \x in {2,4,6} \draw (c4)--(b\x);
 
\end{tikzpicture}
  \caption{Space $\Pp^{2}_2$, finite model of $\Pp^{2}$.}
  \label{fig_finite_P2_2}
  \end{subfigure}
  \caption{Cellular and finite models of $\Pp^{2}$.}
  \label{fig_P2_2}
\end{figure}

To prove the minimality of these finite models we need the following lemmas.

\begin{lemma} \label{lemma_outside_Ua_S1}
Let $X$ be a finite T$_0$--space which does not satisfy condition \Spl[1]. Let $a\in \mxl(X)$. Then there exist two incomparable elements $b_1,b_2 \in X-W^X_a$.

In addition, if $X-W^X_a=\{b_1,b_2\}$ then $U_{b_1}\cap U_{b_2}=\varnothing$.
\end{lemma}

\begin{proof}
Suppose that there do not exist distinct elements $b_1,b_2 \in X-W^X_a$ such that $b_1$ and $b_2$ are incomparable. Then $X-W^X_a$ is a (possibly empty) chain and hence the connected components of $X-V^X_a$ are contractible. Since the connected components of $V^X_a$ are also contractible we see that $X$ satisfies condition \Spl[1].

Now, if $X-W^X_a=\{b_1,b_2\}$ and $z\in U_{b_1}\cap U_{b_2}$ then $(X;V^X_a,F_z\cup\mxl(X))$ is a splitting triad which satisfies \Spl[1].
\end{proof}

\begin{lemma} \label{lemma_relp_no_S1}
Let $X$ be a finite and connected T$_0$--space. If $X$ does not satisfy condition \Spl[1], then for all $(x,y)\in \M$ there exists $b\in\B$ such that $b\,\Relp\,(x,y)$.
\end{lemma}

\begin{proof}
Suppose that there exists $(x,y)\in \M$ such that $\Relp^{-1}(x,y)=\varnothing$ and let $G_y=F_y\cup\mxl(X)$. Then $V_x \cup G_y=X$. Since the connected components of the subspaces $V_x$ and $G_y$ are contractible we obtain that the splitting triad $(X;V_x,G_y)$ satisfies \Spl[1]. Hence, $X$ satisfies \Spl[1].
\end{proof}

\begin{lemma} \label{lemma_relm_card_B_no_S1}
Let $X$ be a finite and connected T$_0$--space without beat points. Suppose that $X$ does not satisfy condition \Spl[1]. Then, 
$$l_X\geq\frac{2m_X}{m_X-2}$$
and the equality holds if and only if $\alpha_b=2$ for all $b\in\B$ and $\#\Relm^{-1}(a)=2$ for all $a\in\mxl(X)$.
\end{lemma}

\begin{proof}
Note that $\#\mxl(X)\geq 3$ and $\#\mnl(X)\geq 3$ by \ref{rem_mxl_and_cond}. Now, by \ref{lemma_outside_Ua_S1}, 
$$2m_X\leq \sum\limits_{a\in\mxl(X)}\Relm^{-1}(a) = \#\Relm = \sum\limits_{b\in \B}(m_X-\alpha_b)\leq l_X(m_X-2)\ .$$
The result follows.
\end{proof}

\begin{lemma} \label{lemma_relp_card_B_no_S1}
Let $X$ be a finite and connected T$_0$--space without beat points. Suppose that $X$ does not satisfy condition \Spl[1]. Then, 
$$l_X\geq\frac{m_X n_X}{(m_X-2)(n_X-2)}$$
and the equality holds only if $\alpha_b=\beta_b=2$ for all $b\in\B$.
\end{lemma}

\begin{proof}
As in the previous proof, $\#\mxl(X)\geq 3$ and $\#\mnl(X)\geq 3$ by \ref{rem_mxl_and_cond}. Now, by \ref{lemma_relp_no_S1},
$$m_X n_X \leq \#\Relp=\sum\limits_{b\in \B}(m_X-\alpha_b)(n_X-\beta_b)\leq l_X(m_X-2)(n_X-2)\ .$$
The result follows.
\end{proof}

\begin{lemma} \label{lemma_relp_card_B_comp}
Let $X$ be a finite and connected T$_0$--space without beat points. Suppose that $X$ does not satisfy condition \Spl[1] and that there exist $c,d \in\B$ such that $c<d$. Then 
$$l_X\geq\frac{m_Xn_X}{(m_X-2)(n_X-2)}+1$$
and the equality holds only if $\alpha_b=\beta_b=2$ for all $b\in\B$.
\end{lemma}

\begin{proof}
As in the previous proofs, $\#\mxl(X)\geq 3$ and $\#\mnl(X)\geq 3$. Since $c<d$, $F_d\subseteq F_c$ and $U_c\subseteq U_d$. Thus, $2\leq \alpha_d\leq \alpha_c\leq m_X$, $2\leq \beta_c\leq \beta_d\leq n_X$ and $c,d\in\Relp^{-1}(x,y)$ for all $x\in\mxl(X)-F_c$ and $y\in \mnl(X)-U_d$.

Therefore, applying \ref{lemma_relp_no_S1} we obtain that $\#\Relp\geq m_X n_X+(m_X-\alpha_c)(n_X-\beta_d)$. On the other hand,
\begin{displaymath}
\begin{array}{rcl}
\#\Relp & = & (m-\alpha_c)(n-\beta_c)+(m-\alpha_d)(n-\beta_d)+\sum\limits_{b\in \B-\{c,d\}}(m-\alpha_b)(n-\beta_b) \leq \\
& \leq & (m-\alpha_c)(n-\beta_c)+(m-\alpha_d)(n-\beta_d) + (l-2)(m-2)(n-2)
\end{array}
\end{displaymath}
Thus,
\begin{displaymath}
\begin{array}{rcl}
(l-2)(m-2)(n-2) & \geq & mn+(m-\alpha_c)(n-\beta_d)- (m-\alpha_c)(n-\beta_c)-(m-\alpha_d)(n-\beta_d) = \\
& = & mn+ (\alpha_d-\alpha_c)(\beta_c-\beta_d) - (m-\alpha_d)(n-\beta_c) \geq \\
& \geq & mn - (m-\alpha_d)(n-\beta_c)  \geq mn - (m-2)(n-2) \ .
\end{array}
\end{displaymath}
The result follows.
\end{proof}

The following theorem is one of the main results of this article.

\begin{theo} \label{theo_13_points}
Let $X$ be a finite and connected T$_0$--space which does not satisfy condition \Spl[1]. Then $\# X\geq 13$. 
\end{theo}

\begin{proof}
By \ref{prop_S123_beat_points} we may assume that $X$ does not have beat points. Since $\#X=m_X+n_X+l_X$, from \ref{lemma_relp_card_B_no_S1} we obtain that $\#X\geq m_X+n_X+\frac{m_Xn_X}{(m_X-2)(n_X-2)}$. Let $a=m_X-2$ and $b=n_X-2$. Then 
\begin{displaymath}
\begin{array}{rcl}
\#X & \geq & a+b+\dfrac{(a+2)(b+2)}{ab}+4 = \left(\dfrac{2}{a}+\dfrac{a}{2}\right)+\left(\dfrac{2}{b}+\dfrac{b}{2}\right)+\left(\dfrac{4}{ab}+\dfrac{a}{2}+\dfrac{b}{2}\right)+5 \geq \\
 & \geq & 2+2+3+5 = 12
\end{array}
\end{displaymath}
where the last inequality holds by the inequality of arithmetic and geometric means, with equality holding if and only if $a=b=2$. Thus, if $m_X\neq 4$ or $n_X\neq 4$ we obtain that $\#X\geq 13$ as desired.

Suppose now that $m_X=n_X=l_X=4$. By \ref{lemma_relp_card_B_comp}, $\B$ must be an antichain and from \ref{lemma_relp_card_B_no_S1} we obtain that $\alpha_b=\beta_b=2$ for all $b\in\B$. From \ref{lemma_relm_card_B_no_S1} it follows that $\#\Relm^{-1}(a)=2$ for all $a\in\mxl(X)$. Thus, $\#(U_x\cap\B)= 2$ for all $x\in\mxl(X)$. Let $x_1\in\mxl(X)$ and suppose that $U_{x_1}\cap\B=\{b_1,b_2\}$ and $\B-U_{x_1}=\{b_3,b_4\}$. Note that $1\leq \#(\widehat{F}_{b_3}\cap \widehat{F}_{b_4})\leq 2$. If $\#(\widehat{F}_{b_3}\cap \widehat{F}_{b_4})=1$ then $X-V_{x_1}$ is contractible and hence $X$ satisfies \Spl[1], which contradicts the hypothesis.

Thus, $\#(\widehat{F}_{b_3}\cap \widehat{F}_{b_4})=2$. Then, there exist $x_3,x_4\in \mxl(X)$ such that $\widehat{F}_{b_3}\cap \widehat{F}_{b_4}=\{x_3,x_4\}$ and hence $\widehat{F}_{b_3} = \widehat{F}_{b_4}=\{x_3,x_4\}$. Clearly $x_3$ and $x_4$ are different from $x_1$. Let $x_2 \in \mxl(X)-\{x_1,x_3,x_4\}$. Thus, $b_3,b_4\notin U_{x_2}$. Hence, $U_{x_2}\cap \B=\{b_1,b_2\}$.

Summing up, we have $U_{x_1}\cap \B=U_{x_2}\cap \B=\{b_1,b_2\}$ and $U_{x_3}\cap \B=U_{x_4}\cap \B=\{b_3,b_4\}$. That is to say that $X-\mnl(X)$ is homeomorphic to $\Su D_2 \amalg \Su D_2$. In a similar way we obtain that $X-\mxl(X)$ is homeomorphic to $\Su D_2 \amalg \Su D_2$. And since $X$ is connected, without loss of generality we may assume that $\widehat U_{b_1}=\widehat U_{b_3}=\{y_1,y_3\}$ and $\widehat U_{b_2}=\widehat U_{b_4}=\{y_2,y_4\}$, where $y_1$, $y_2$, $y_3$ and $y_4$ are the minimal points of $X$.

Now note that the connected components of $\{x_1,x_2,x_3,x_4,b_1,b_4\}$ and $\{b_2,b_3,y_1,y_2,y_3,y_4\}$ are contractible and hence $X$ satisfies \Spl[1], which contradicts the hypothesis.
\end{proof}

As a corollary of the previous theorem we obtain a purely mathematical proof for a result of Adamaszek which states that the first homology group of any poset with at most 12 elements is torsion-free \cite{Ad}. We ought to mention that Adamaszek gave a computer-aided proof which needed to analyse approximately 90 million cases and took more than 200 hours of processor time.

\begin{theo}\label{coro_Adamaszek}
Let $X$ be a finite T$_0$--space such that $\# X\leq 12$. Then $\pi_1(X)$ is a free group. In particular, $H_1(X)$ is torsion-free.
\end{theo}

\begin{proof}
Without loss of generality, we may assume that $X$ is connected. By \ref{theo_13_points}, $X$ satisfies \Spl[1]. Thus, by \ref{theo_splitting_free_fundamental_group}, $\pi_1(X)$ is a free group. Hence, $H_1(X)$ is torsion-free.
\end{proof}

As a corollary of the results above we obtain an affirmative answer to the question about the minimality of a finite model of the real projective plane raised in \cite[Example 7.1.1]{BarLN}.

\begin{coro} \label{coro_real_proj_plane}
Let $X$ be a finite model of the real projective plane. Then $\# X\geq 13$.
\end{coro}

We will now prove a conjecture of Hardie, Vermeulen and Witbooi which states that posets with fewer than 13 points do not have 2--torsion in integral homology (cf. \cite[Remark 4.2]{HVW}). Indeed, we will prove the following stronger result.

\begin{theo}\label{theo_HVW}
Let $X$ be a finite T$_0$--space such that $H_k(X)$ has torsion for some $k\in\N$. Then $\# X\geq 13$.
\end{theo}

\begin{proof}
We may assume that $X$ does not have beat points and that the integral homology groups of every proper subspace of $X$ are torsion-free. Choose $k\in\N$ such that $H_k(X)$ has torsion. By \ref{coro_Adamaszek} we may assume that $k\geq 2$.

Let $p\in X$. From \cite[Proposition 3.2]{CO} we obtain that there is a long exact sequence
\begin{displaymath}
\xymatrix{\ldots \ar[r] & H_k(\widehat C_p) \ar[r] & H_k(X-\{p\}) \ar[r] & H_k(X) \ar[r] & H_{k-1}(\widehat C_p) \ar[r] & \ldots}
\end{displaymath}
If $H_k(\widehat C_p)=0$, since $H_k(X-\{p\})$ and $H_{k-1}(\widehat C_p)$ are free abelian by our assumptions we obtain that $H_k(X)$ is free abelian. Thus, $H_k(\widehat C_p)\neq 0$. In particular, $h(C_p)\geq k+1$ and hence every point $p\in X$ belongs to a chain of height $k+1$.

Let $a\in\mxl(X)$. Suppose that $X-U_a$ is an antichain. Then, by \cite[Proposition 3.2]{CO} 
$$H_k(X)=H_k(X,U_a)=\bigoplus_{x\in X-U_a}H_{k-1}(\widehat C_x)$$
which is torsion-free by one of the assumptions above. Hence, $X-U_a$ is not an antichain. Thus, $X-U_a-\mxl(X)\neq\varnothing$.

Let $d\in  \mxl(X-U_a-\mxl(X))$. By \ref{rem_minimal_space_two_points}, there exists $b,c\in\mxl(X)$ such that $b\neq c$, $b\geq d$ and $c\geq d$. Note that $b\neq a$ and $c\neq a$. Applying the same reasoning as above we obtain that $X-U_b$ is not an antichain and thus there exists a point $e\in \mxl(X-U_b-\mxl(X))$. Again, by \ref{rem_minimal_space_two_points} we obtain that $\# (F_e\cap\mxl(X))\geq 2$.
Let $X_1=\mxl(X)\cup\mxl(X-\mxl(X))$. If $F_e\cap\mxl(X)\nsubseteq \{a,b,c\}$ then $\#\mxl(X)\geq 4$ and thus $\#X_1 \geq 6$. Otherwise, $\widehat F_e=\{a,c\}$ and since $X-U_c$ is not an antichain, there exists a point $f\in \mxl( X-U_c-\mxl(X))$ which must, of course, be different from $d$ and $e$. Therefore, $\#X_1 \geq 6$ in any case.

Let $X_2=\mnl(X)\cup\mnl(X-\mnl(X))$. Applying the argument above to $X^\op$, we obtain that $\#X_2 \geq 6$. Since $H_k(X)$ has torsion it follows that $h(X)\geq k+1 \geq 3$. And since every point of $X$ belongs to a chain of height $k+1$ we obtain that $X_1\cap X_2 = \varnothing$. Therefore, $\# X\geq 12$.

Suppose now that $\# X= 12$. Then $X=X_1\cup X_2$ and $\#X_1=\#X_2 =6$. In particular, $h(X)=3$ and then $k=2$. Following the previous argument and with the notation used before, we have that either $F_e\cap\mxl(X)\nsubseteq \{a,b,c\}$ or $\widehat F_e=\{a,c\}$.

Suppose that $F_e\cap\mxl(X)\nsubseteq \{a,b,c\}$. It follows that $\#\mxl(X)=4$. Let $g$ be the only element of $F_e \cap \mxl(X)-\{a,b,c\}$. Suppose that $d\leq g$. Since $X-U_g$ is not an antichain, there exists a point $h\in \mxl(X-U_g-\mxl(X))$ which is clearly different from $d$ and $e$. Note that $h\in X_1$. Therefore, $\#X_1 \geq 7$ which entails a contradiction. Therefore $d\not\leq g$. In a similar way we obtain that $e\not\leq c$.

Thus, $\widehat{F}_e\subseteq \{a,g\}$ and hence, by \ref{rem_minimal_space_two_points}, $\widehat{F}_e = \{a,g\}$. Therefore, $X_1$ is the following poset
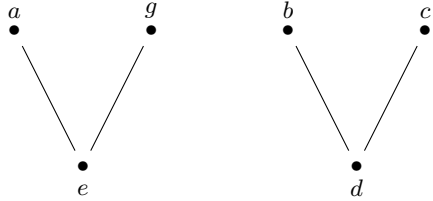
\begin{displaymath}
\begin{tikzpicture}[x=1.8cm,y=1.8cm]
\tikzstyle{every node}=[font=\footnotesize]
\draw (0,1) node(a){$\bullet$} node  [above=1] {$a$};
\draw (1,1) node(g){$\bullet$} node [above=1]{$g$};
\draw (2,1) node(b){$\bullet$} node  [above=1] {$b$};
\draw (3,1) node(c){$\bullet$} node  [above=1] {$c$};
\draw (2.5,0) node(d){$\bullet$} node [below=1]{$d$};
\draw (0.5,0) node(e){$\bullet$} node [below=3]{$e$};
\draw (d) -- (b);
\draw (d) -- (c);
\draw (e) -- (a);
\draw (e) -- (g);
\end{tikzpicture}
\end{displaymath}
which will be called $V$.

But, in this case, if $y$ is a maximal point of $X_2$ then $\varnothing \neq \widehat F_y \subseteq X_1$ and then the connected components of $\widehat F_y$ are contractible (since every subspace of $V$ satisfies this). Thus, $\widehat C_y=\widehat U_y\circledast \widehat F_y$ is homotopy equivalent to the non-Hausdorff join of two discrete spaces and hence $H_2(\widehat C_y)=0$, which entails a contradiction.

Therefore, $\widehat F_e=\{a,c\}$ and from \ref{rem_minimal_space_two_points} we obtain that $\widehat F_f=\{a,b\}$. Then $X_1$ is the following poset
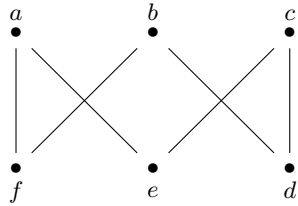
\begin{displaymath}
\begin{tikzpicture}[x=1.8cm,y=1.8cm]
\tikzstyle{every node}=[font=\footnotesize]
\draw (0,1) node(a){$\bullet$} node  [above=1] {$a$};
\draw (1,1) node(b){$\bullet$} node  [above=1] {$b$};
\draw (2,1) node(c){$\bullet$} node  [above=1] {$c$};
\draw (2,0) node(d){$\bullet$} node [below=1]{$d$};
\draw (1,0) node(e){$\bullet$} node [below=3]{$e$};
\draw (0,0) node(f){$\bullet$} node [below=1]{$f$};
\draw (d) -- (b);
\draw (d) -- (c);
\draw (e) -- (a);
\draw (e) -- (c);
\draw (f) -- (a);
\draw (f) -- (b);
\end{tikzpicture}
\end{displaymath}
which will be called $H$.

Hence, the poset $X_1$ is homeomorphic to $H$ and thus, the poset $X_2$ is homeomorphic to $H^\op$ which is homeomorphic to $H$. By the previous argument, if $y$ is a maximal point of $X_2$ then $\widehat F_y$ can not be homotopy equivalent to a discrete space. And since $\widehat F_y\subseteq X_1$ the only possibility is that $\widehat F_y = X_1$ for all $y \in \mxl (X_2)$. It follows that $X=X_2\circledast X_1$, that is, $X$ is the following poset
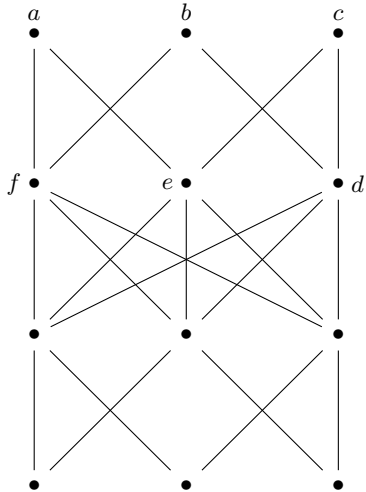
\begin{displaymath}
\begin{tikzpicture}[x=2cm,y=2cm]
\tikzstyle{every node}=[font=\footnotesize]
\draw (0,3) node(a){$\bullet$} node  [above=1] {$a$};
\draw (1,3) node(b){$\bullet$} node  [above=1] {$b$};
\draw (2,3) node(c){$\bullet$} node  [above=1] {$c$};
\draw (2,2) node(d){$\bullet$} node [right=1]{$d$};
\draw (1,2) node(e){$\bullet$} node [left=1]{$e$};
\draw (0,2) node(f){$\bullet$} node [left=1]{$f$};
\draw (0,1) node(f2){$\bullet$};
\draw (1,1) node(e2){$\bullet$};
\draw (2,1) node(d2){$\bullet$};
\draw (2,0) node(c2){$\bullet$};
\draw (1,0) node(b2){$\bullet$};
\draw (0,0) node(a2){$\bullet$};
\draw (d) -- (b);
\draw (d) -- (c);
\draw (e) -- (a);
\draw (e) -- (c);
\draw (f) -- (a);
\draw (f) -- (b);
\foreach \x in {d,e,f} \foreach \y in {f2,e2,d2} \draw (\y) -- (\x);
\draw (d2) -- (b2);
\draw (d2) -- (c2);
\draw (e2) -- (a2);
\draw (e2) -- (c2);
\draw (f2) -- (a2);
\draw (f2) -- (b2);
\end{tikzpicture}
\end{displaymath}
Let $U=U_a\cup U_c$. It is easy to prove that $U$ is contractible. Then, by \cite[Proposition 3.2]{CO},
$$H_2(X)\cong H_2(X,U) \cong H_{1}(\widehat C_b)$$
which is free abelian by our initial assumptions. Thus, this case is not possible either. Therefore, $\# X\geq 13$.
\end{proof}

Now we will prove that there are just two finite models of $\Pp^2$ with 13 points: the spaces $\Pp^2_1$ and $\Pp^2_2$ of figures \ref{fig_finite_P2_1} and \ref{fig_finite_P2_2} respectively.

\begin{theo}
Let $X$ be a finite T$_0$--space such that $X$ is a finite model of the real projective plane and $\#X=13$. Then $X$ is homeomorphic to either $\Pp^2_1$ or $\Pp^2_2$.
\end{theo}

\begin{proof}
By \ref{coro_real_proj_plane}, $X$ does not have beat points. Also, by considering the opposite space $X^{\op}$, we may assume that $\#\mxl(X)\leq \#\mnl(X)$. Note also that $\#\mxl(X)\geq 3$ by \ref{rem_mxl_and_cond} and that $\mxl(X)\cap\mnl(X)=\varnothing$ since $X$ is connected and $\# X \geq 2$.

We will consider several cases.

\underline{Case 1}: $\#\mxl(X)=3$.

\underline{Case 1.1}: $\#\mnl(X)=3$. By \ref{lemma_relp_card_B_no_S1}, $l_X\geq 9$ and thus $\#X\geq 15$.

\underline{Case 1.2}: $\#\mnl(X)\geq 4$. Suppose $\mxl(X)=\{a_1,a_2,a_3\}$. By \ref{lemma_outside_Ua_S1}, there exist incomparable points $b_1,b_2 \in X-W_{a_3}^X$. By \ref{rem_minimal_space_two_points}, for $i\in\{1,2\}$ we have that $\widehat{F}_{b_i}\cap\mxl(X)\geq 2$ and since $b_i\nleq a_3$ we obtain that $\widehat{F}_{b_i}\cap\mxl(X)=\{a_1,a_2\}$. Thus, $b_1,b_2\in U_{a_1}\cap U_{a_2}-U_{a_3}$.

In a similar fashion, we obtain that there exist incomparable points $b_3,b_4\in U_{a_1}\cap U_{a_3}-U_{a_2}$ and incomparable points $b_5,b_6\in U_{a_2}\cap U_{a_3}-U_{a_1}$ which are not minimal in $X$. Note that $\{b_1,b_2,b_3,b_4,b_5,b_6\}$ is an antichain with six elements.

Since $\#X=13$, we obtain that $\#\mnl(X)= 4$. Let $\mnl(X)=\{c_1,c_2,c_3,c_4\}$. Since $X-W_{a_3}^X=\{b_1,b_2\}$, from \ref{lemma_outside_Ua_S1} it follows that $U_{b_1}\cap U_{b_2}=\varnothing$. Therefore, applying \ref{rem_minimal_space_two_points} we obtain that $\#(U_{b_1}\cap \mnl(X))=\#(U_{b_2}\cap \mnl(X))=2$ and $U_{b_1}\cup U_{b_2}=\mnl(X)$ and thus $F_{c_1}$ contains exactly one of the points $b_1$, $b_2$. A similar result holds for the points $b_3$ and $b_4$ and for the points $b_5$ and $b_6$.

Without loss of generality, we may assume that $b_2$, $b_4$ and $b_6$ belong to $F_{c_4}$. We will prove now that $\#(U_{b_1}\cap U_{b_3}) \neq 2$. Suppose that $\#(U_{b_1}\cap U_{b_3}) = 2$. Without loss of generality, we may assume that $U_{b_1}\cap U_{b_3}=\{c_1,c_2\}$. If $c_3\leq b_5$ then $X-F_{c_4}=\{b_1,b_3,b_5,c_1,c_2,c_3\}$ is homotopy equivalent to $\{b_1,b_3,c_1,c_2\}$ which is included in the contractible subspace $\{b_1,b_3,c_1,c_2,a_1\}$. Thus, the inclusion of $X-F_{c_4}$ in $X$ induces the trivial map between the fundamental groups and hence $X$ satisfies \Spl[1]. If $c_3 \nleq b_5$ then $c_1\leq b_5$ and $c_2\leq b_5$ and thus the connected components of $X-F_{c_4}$ are $\{c_3\}$ and $\{b_1,b_3,b_5,c_1,c_2\}$. Note that the latter is included in the contractible subspace $\{b_1,b_3,b_5,c_1,c_2,a_1,a_3\}$. Hence, $X$ satisfies \Spl[1]. Therefore, $\#(U_{b_1}\cap U_{b_3}) \neq 2$.

Hence, $\#(U_{b_1}\cap U_{b_3})=1$ since $\widehat U_{b_1}$ and $\widehat U_{b_3}$ are included in $\{c_1,c_2,c_3\}$. In a similar way, $\#(U_{b_1}\cap U_{b_5})=1$ and $\#(U_{b_3}\cap U_{b_5})=1$. Hence, without loss of generality we may assume that 
$\widehat U_{b_1}=\{c_1,c_2\}$, $\widehat U_{b_3}=\{c_1,c_3\}$ and $\widehat U_{b_5}=\{c_2,c_3\}$. It follows that $\widehat U_{b_2}=\{c_3,c_4\}$, $\widehat U_{b_4}=\{c_2,c_4\}$ and $\widehat U_{b_6}=\{c_1,c_4\}$. Note that in this case $X$ is homeomorphic to the space $\Pp^2_2$ of figure \ref{fig_finite_P2_2}.

\underline{Case 2}: $\#\mxl(X)=4$. Let $a_1$, $a_2$, $a_3$ and $a_4$ be the maximal points of $X$.

\underline{Case 2.1}: $\#\mnl(X)=4$. Then $\#\B=5$. If there exist $d_1,d_2\in\B$ such that $d_1<d_2$ then, from \ref{lemma_relp_card_B_comp}, we obtain that $\alpha_b=\beta_b=2$ for all $b\in\B$. Hence, $F_{d_1}\cap\mxl(X)=F_{d_2}\cap\mxl(X)$. Without loss of generality, we may assume that $F_{d_1}\cap\mxl(X)=\{a_1,a_2\}$.

Since $X$ does not have beat points, $\widehat F_{d_1}$ does not have minimum element. Thus, $\#\mnl(\widehat F_{d_1})\geq 2$ and it is easy to prove that $\mnl(\widehat F_{d_1})\subseteq \B$. Let $e_1$, $e_2$ be distinct elements of $\mnl(\widehat F_{d_1})$. Note that $F_{e_1}\cap\mxl(X)=F_{e_2}\cap\mxl(X)=\{a_1,a_2\}$. In a similar way, $\#\mxl(\widehat U_{e_1})\geq 2$ and $\mxl(\widehat U_{e_1})\subseteq \B$. Thus, $\#(U_{a_1}\cap \B) \geq 4$ which contradicts \ref{lemma_outside_Ua_S1}.

Therefore, $\B$ must be an antichain.

Since $\#(F_b\cap\mxl(X))\geq 2$ for all $b\in \B$ it follows that there exists $a\in\mxl(X)$ such that $\#(U_a\cap \B)\geq 3$. Without loss of generality, we may assume that $\#(U_{a_1}\cap \B)\geq 3$. By \ref{lemma_outside_Ua_S1}, there exist incomparable points $b_1,b_2\in X-W_{a_1}^X$. Hence, $\#(U_{a_1}\cap \B)= 3$ and $\B-U_{a_1}=\{b_1,b_2\}$. Thus, $U_{b_1}\cap U_{b_2}=\varnothing$ by \ref{lemma_outside_Ua_S1}. Hence, from \ref{rem_minimal_space_two_points} it follows that $\#(U_{b_1}\cap\mnl(X))=2$ and $\#(U_{b_2}\cap\mnl(X))=2$.

Now note that, by \ref{rem_minimal_space_two_points}, $F_{b_1}\cap F_{b_2}\neq \varnothing$. If $\#(F_{b_1}\cap F_{b_2})=1$ then $X-V_{a_1}^X$ is contractible and hence $X$ satisfies \Spl[1]. Thus $\#(F_{b_1}\cap F_{b_2})\geq 2$. Without loss of generality we may assume that $\{a_3,a_4\}\subseteq F_{b_1}\cap F_{b_2}$.

It follows that $U_{a_j}\supseteq \mnl(X)$ for $j\in\{3,4\}$. We will prove now that the same holds for $j\in\{1,2\}$. Note that $\#(U_a\cap \B)\leq 3$ for all $a\in\mxl(X)$ by \ref{lemma_outside_Ua_S1}. If $\#(U_{a_k}\cap \B)=3$ for some $k\in\{3,4\}$ then applying the reasoning above to $a_k$ shows that $U_{a_j}\supseteq \mnl(X)$ for $j\in\{1,2\}$. Otherwise, $U_{a_j}\cap \B=\{b_1,b_2\}$ for $j\in\{3,4\}$ and hence $\B - \{b_1,b_2\}\subseteq U_{a_2}$. Let $b_3,b_4,b_5\in \B$ such that $\B - \{b_1,b_2\}=\{b_3,b_4,b_5\}$. If $\#(U_{b_3}\cup U_{b_4}\cup U_{b_5})\leq 3$ then either there exists $z\in \mnl(X)$ such that $z\in U_{b_3}\cap U_{b_4}\cap U_{b_5}$, or there exist distinct points $z_1,z_2 \in \mnl(X)$ such that each of them belongs to exactly two of the subsets $U_{b_3}$, $U_{b_4}$ and $U_{b_5}$ and such that $\{b_3,b_4,b_5\}\subseteq F_{z_1}\cup F_{z_2}$. Note that in both cases the subspace $\{a_1,a_2,b_3,b_4,b_5\}$ is included in a contractible subspace of $X$. It follows that the splitting triad $(X;X-V_{a_4},V_{a_4})$ satisfies \Spl[1] which entails a contradiction. Therefore, $U_a\supseteq \mnl(X)$ for all $a\in \mxl(X)$.

Now, note that if $b\in\B$ is such that $\alpha_b\geq 3$ and $\beta_b\geq 3$ then the connected components of $X-C_b$ are contractible and hence $X$ satisfies \Spl[1]. Thus, for all $b\in\B$ either $\alpha_b=2$ or $\beta_b=2$. Hence, $(\alpha_b-1)(\beta_b-1)=3-\frac{1}{2}(4-\alpha_b)(4-\beta_b)$ for all $b\in\B$.

As in the proof of \ref{lemma_relp_card_B_no_S1} we obtain that $16\leq \#\Relp=\sum\limits_{b\in \B}(4-\alpha_b)(4-\beta_b)$. Thus, it follows that the Euler characteristic of $X$ is
\begin{displaymath}
\begin{array}{rcl}
\chi(X) & = & \displaystyle 13-\left(\sum_{b\in \B}\alpha_b + \sum_{b\in \B}\beta_b + 16\right)+\sum_{b\in \B}\alpha_b\beta_b = \\
& = & \displaystyle -8+ \sum_{b\in \B}(\alpha_b-1)(\beta_b-1) = -8+ \sum_{b\in \B} \left(3-\frac{1}{2}(4-\alpha_b)(4-\beta_b)\right) = \\ 
& = & \displaystyle 7-\frac{1}{2}\sum_{b\in \B} (4-\alpha_b)(4-\beta_b) \leq -1 .
\end{array}
\end{displaymath}
Hence, $X$ can not be a finite model of the real projective plane.

\underline{Case 2.2}: $\#\mnl(X)\geq 5$. From \ref{lemma_relm_card_B_no_S1} it follows that $\#\B=4$ and that $\alpha_b=2$ for all $b\in\B$ and $\#\Relm^{-1}(a)=2$ for all $a\in\mxl(X)$. Since $\#\B=4$ we obtain that $\#\mnl(X)= 5$. From \ref{lemma_relp_card_B_comp} it follows that $\B$ must be an antichain. Proceeding as in the proof of \ref{theo_13_points} we obtain that $X-\mnl(X)$ is homeomorphic to $\Su D_2 \amalg \Su D_2$.

Without loss of generality we may assume that $U_{a_1}\cap \B=U_{a_2}\cap \B=\{b_1,b_2\}$ and $U_{a_3}\cap \B=U_{a_4}\cap \B=\{b_3,b_4\}$. By \ref{lemma_outside_Ua_S1}, $U_{b_1}\cap U_{b_2}= \varnothing$ and $U_{b_3}\cap U_{b_4}= \varnothing$.

Let $i\in\{1,2\}$ and $j\in\{3,4\}$. If $\#(U_{b_i}\cap U_{b_j})\leq 1$ then the connected components of $\{b_i,b_j\}\cup\mnl(X)$ and its complement are contractible and hence $X$ satisfies \Spl[1]. Therefore $\#(U_{b_i}\cap U_{b_j})\geq 2$ for all $i\in\{1,2\}$ and $j\in\{3,4\}$. And since $U_{b_3}\cap U_{b_4}= \varnothing$ we obtain that $\#\widehat U_{b_i}\geq 4$ for all $i\in\{1,2\}$. Thus, $U_{b_1}\cap U_{b_2}\neq \varnothing$, a contradiction.

\underline{Case 3}: $\#\mxl(X)\geq 5$. Let $a_1\in\mxl(X)$. By \ref{lemma_outside_Ua_S1}, there exist incomparable points $b_1,b_2\in X-W_{a_1}^X$. Moreover, we may assume that $F_{b_1}\cap F_{b_2}\neq \varnothing$ since $F_y\cap F_z=\varnothing$ for all $y,z\in X-W_{a_1}^X$ implies that the connected components of $X-V_{a_1}^X$ are contractible and hence $X$ satisfies \Spl[1].

Let $a_2\in\mxl(X)$ such that $a_2\in F_{b_1}\cap F_{b_2}$. By \ref{lemma_outside_Ua_S1}, there exist incomparable points $b_3,b_4\in X-W_{a_2}^X$ which must, of course, be different from $b_1$ and $b_2$. Therefore $\#X\geq 14$.
\end{proof}

\section{Minimal finite models of the torus and the Klein bottle} \label{section_torus}

In this section we will prove that the space $\Su S^0\times \Su S^0$ is a minimal finite model of the torus, answering another open question of Barmak \cite[p. 44]{BarLN}. This will be obtained as a corollary of theorem \ref{theo_16_points} which states that if a poset $X$ does not satisfy \Spl[2] then $\#X\geq 16$. From this theorem we will also infer that there do not exist finite models of the Klein bottle with fewer than 16 points.

With further development of our techniques we will also find all the minimal finite models of the torus and the Klein bottle, which is a much harder task. Indeed, we will prove that there exist exactly two minimal finite models of the torus: the spaces $\T_{0,0}$ and $\T_{1,1}$ of figures \ref{fig_finite_Q00} and \ref{fig_finite_Q11}. It is worth mentioning that the space $\T_{1,1}$ was not previously known and appears naturally in the proof of \ref{theo_torus_Klein}. In the same proof, we will also show that there exist exactly two minimal finite models of the Klein bottle: the space $\Kl_{1,0}$ of figure \ref{fig_finite_Q10} and its opposite, the space $\Kl_{0,1}$ of figure \ref{fig_finite_Q01}.

We will begin by describing the aforementioned finite models of the torus and the Klein bottle. Consider the regular CW-complex structure of the torus $T^{2}$ given in figure \ref{fig_CW_Q00} with 2--cells $a_i$, 1--cells $b_j$, and 0--cells $c_k$, where the 0--cells and the 1--cells are identified in the standard way. The face poset of this regular CW-complex, given in figure \ref{fig_finite_Q00}, is a finite model of $T^{2}$. 

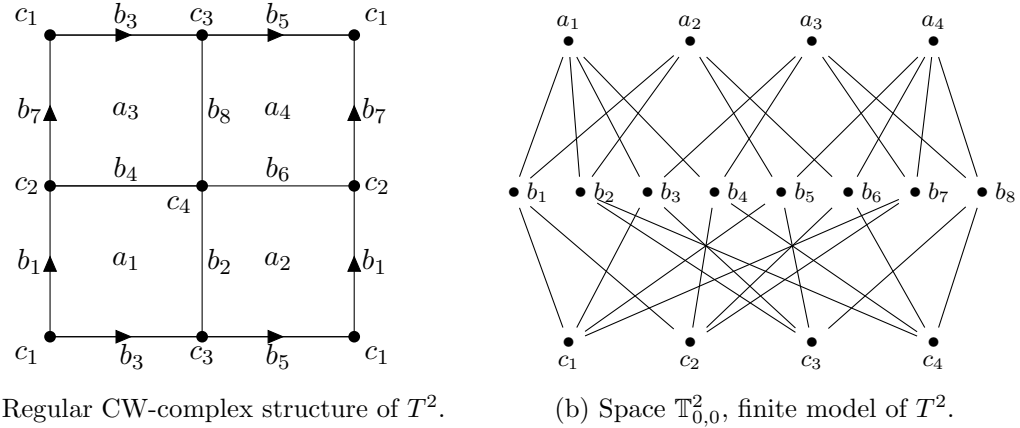
\begin{figure}[h!]
\label{fig_Q00}
  \begin{subfigure}[t]{0.45\textwidth}
  \begin{center}
    \begin{tikzpicture}[x=4cm,y=4cm]
	\tikzset{->->-/.style={decoration={
	markings,
	mark=at position 0.27 with {\arrow{triangle 45}},
	mark=at position 0.77 with {\arrow{triangle 45}}},postaction={decorate}}}
	\tikzset{->-/.style={decoration={
	markings,
	mark=at position 0.57 with {\arrow{triangle 45}}},postaction={decorate}}}
	\tikzstyle{every node}=[font=\normalsize]

	\draw(0,0) node[below left]{$c_1$} -- (0.5,0);
	\draw(0.5,0) node[below]{$c_3$}--	(1,0) node[below right]{$c_1$};
	\draw(0,0.5) node[left]{$c_2$} --	(0.5,0.5) node[below left]{$c_4$};
	\draw(0,0.5)--(1,0.5) node[right]{$c_2$};
	\draw (0,1) node[above left]{$c_1$} --(0.5,1);
	\draw(0.5,1) node[above]{$c_3$}--	(1,1) node[above right]{$c_1$};
	
	\draw (0.5,0) -- (0.5,1);

	\draw[->->-] 	(0,0) -- (1,0);
	\draw[->->-] (0,1) -- (1,1);
	\draw[->->-] 	(0,0) -- (0,1);
	\draw[->->-] (1,0) -- (1,1);

	\fill[fill=black] (0,0) circle (.2em);
	\fill[fill=black] (0,0.5) circle (.2em);
	\fill[fill=black] (0,1) circle (.2em);
	\fill[fill=black] (0.5,0) circle (.2em);
	\fill[fill=black] (0.5,0.5) circle (.2em);
	\fill[fill=black] (0.5,1) circle (.2em);
	\fill[fill=black] (1,0) circle (.2em);
	\fill[fill=black] (1,0.5) circle (.2em);
	\fill[fill=black] (1,1) circle (.2em);
	
	\draw(0,0.25) node[left=-1] {$b_1$};
	\draw(0,0.75) node[left=-1] {$b_7$};
	\draw(1,0.25) node[right=-1] {$b_1$};
	\draw(1,0.75) node[right=-1] {$b_7$};
	\draw(0.27,0) node[below=-1] {$b_3$};
	\draw(0.75,0) node[below=-1] {$b_5$};
	\draw(0.25,1) node[above=-1] {$b_3$};
	\draw(0.75,1) node[above=-1] {$b_5$};
	\draw (0.5,0.25) node[right=-2]{$b_2$};
	\draw (0.5,0.75) node[right=-2]{$b_8$};
	\draw (0.25,0.5) node[above=-2]{$b_4$};
	\draw (0.75,0.5) node[above=-2]{$b_6$};
	\draw (0.25,0.25) node{$a_1$};
	\draw (0.25,0.75) node{$a_3$};
	\draw (0.75,0.25) node{$a_2$};
	\draw (0.75,0.75) node{$a_4$};
      \end{tikzpicture}
    \caption{Regular CW-complex structure of $T^{2}$.}
    \label{fig_CW_Q00}
      \end{center}
  \end{subfigure}
  \quad
  \begin{subfigure}[t]{0.45\textwidth}
    \begin{tikzpicture}[x=4cm,y=4cm]
    
\tikzstyle{every node}=[font=\footnotesize]

\foreach \x in {1,...,4} \draw (0.4*\x,1) node(a\x){$\bullet$} node[above=1]{$a_{\x}$};
\foreach \x in {1,...,8} \draw (0.22*\x,0.5) node(b\x){$\bullet$} node[right=1]{$b_{\x}$};
\foreach \x in {1,...,4} \draw (0.4*\x,0) node(c\x){$\bullet$} node[below=1]{$c_{\x}$};

\foreach \x in {1,2,3,4} \draw (a1)--(b\x);
\foreach \x in {1,2,5,6} \draw (a2)--(b\x);
\foreach \x in {3,4,7,8} \draw (a3)--(b\x);
\foreach \x in {5,6,7,8} \draw (a4)--(b\x);

\foreach \x in {1,3,5,7} \draw (c1)--(b\x);
\foreach \x in {1,4,6,7} \draw (c2)--(b\x);
\foreach \x in {2,3,5,8} \draw (c3)--(b\x);
\foreach \x in {2,4,6,8} \draw (c4)--(b\x);

\end{tikzpicture}
  \caption{Space $\T_{0,0}$, finite model of $T^{2}$.}
  \label{fig_finite_Q00}
  \end{subfigure}
  \caption{Cellular and finite models of $T^{2}$.}

\end{figure}

We will refer to this space as $\T_{0,0}$, name that follows from the proof of theorem \ref{theo_torus_Klein} below. This finite space is homeomorphic to $\Su S^0\times \Su S^0$.

In a similar way, we can construct the finite model of the Klein bottle of figure \ref{fig_finite_Q10} as the face poset of the regular CW-complex structure of figure \ref{fig_CW_Q10}. The Klein bottle will be denoted by $\Kl$.

\begin{figure}[h!]
  \begin{subfigure}[t]{0.45\textwidth}
  \begin{center}
 \begin{tikzpicture}[x=4cm,y=4cm]
   \tikzstyle{every node}=[font=\normalsize]

  \draw[color=black]
    (0,0) node[below left]{$c_1$} --
	(0.5,0) node[below]{$c_3$}--
	(1,0) node[below right]{$c_1$};
  \draw[color=black]
    (0,0.5) node[left]{$c_2$} --
	(0.5,0.5) node[below left]{$c_4$}--
	(1,0.5) node[right]{$c_2$};
  \draw[color=black]
    (0,1) node[above left]{$c_1$} --
	(0.5,1) node[above]{$c_3$}--
	(1,1) node[above right]{$c_1$};
  \draw 	(0,0) -- (0,1);
  \draw (0.5,0) -- (0.5,1);
  \draw (1,0) -- (1,1);

 \draw(0,0.25) node[left=-1] {$b_4$};
 \draw(0,0.75) node[left=-1] {$b_5$};
 \draw(1,0.25) node[right=-1] {$b_5$};
 \draw(1,0.75) node[right=-1] {$b_4$};

 \draw(0.27,0) node[below=-1] {$b_1$};
 \draw(0.75,0) node[below=-1] {$b_7$};
 \draw(0.25,1) node[above=-1] {$b_1$};
 \draw(0.75,1) node[above=-1] {$b_7$};

\draw (0.5,0.25) node[right=-1]{$b_3$};
\draw (0.5,0.75) node[right=-1]{$b_6$};
\draw (0.25,0.5) node[above=-1]{$b_2$};
\draw (0.75,0.5) node[above=-1]{$b_8$};
\draw (0.25,0.25) node{$a_1$};
\draw (0.25,0.75) node{$a_2$};
\draw (0.75,0.25) node{$a_3$};
\draw (0.75,0.75) node{$a_4$};

 \fill[fill=black] (0,0) circle (.2em);
 \fill[fill=black] (0,0.5) circle (.2em);
 \fill[fill=black] (0,1) circle (.2em);
 \fill[fill=black] (0.5,0) circle (.2em);
 \fill[fill=black] (0.5,0.5) circle (.2em);
 \fill[fill=black] (0.5,1) circle (.2em);
 \fill[fill=black] (1,0) circle (.2em);
 \fill[fill=black] (1,0.5) circle (.2em);
 \fill[fill=black] (1,1) circle (.2em);

 \draw[\arr] (0,0) -- (0,0.28); 
 \draw[\arr] (0,0) -- (0,0.78); 
 \draw[\arr] (1,0.28) -- (1,0.22);
 \draw[\arr] (1,0.78) -- (1,0.72);

 \draw[\arr] (0,0) -- (0.28,0);
 \draw[\arr] (0,0) -- (0.78,0);
 \draw[\arr] (0,1) -- (0.28,1);
 \draw[\arr] (0,1) -- (0.78,1);

\end{tikzpicture}
    \caption{Regular CW-complex structure of $\Kl$.}
    \label{fig_CW_Q10}
      \end{center}
  \end{subfigure}
  \quad
  \begin{subfigure}[t]{0.45\textwidth}
  \begin{tikzpicture}[x=4cm,y=4cm]
   \tikzstyle{every node}=[font=\footnotesize]

\foreach \x in {1,...,4} \draw (0.4*\x,1) node(a\x){$\bullet$} node[above=1]{$a_{\x}$};
\foreach \x in {1,...,8} \draw (0.22*\x,0.5) node(b\x){$\bullet$} node[right=1]{$b_{\x}$};
\foreach \x in {1,...,4} \draw (0.4*\x,0) node(c\x){$\bullet$} node[below=1]{$c_{\x}$};

\foreach \x in {1,2,3,4} \draw (a1)--(b\x);
\foreach \x in {1,2,5,6} \draw (a2)--(b\x);
\foreach \x in {3,5,7,8} \draw (a3)--(b\x);
\foreach \x in {4,6,7,8} \draw (a4)--(b\x);

\foreach \x in {1,4,5,7} \draw (c1)--(b\x);
\foreach \x in {2,4,5,8} \draw (c2)--(b\x);
\foreach \x in {1,3,6,7} \draw (c3)--(b\x);
\foreach \x in {2,3,6,8} \draw (c4)--(b\x);

\end{tikzpicture}
 
  \caption{Space $\Kl_{1,0}$, finite model of $\Kl$.}
  \label{fig_finite_Q10}
  \end{subfigure}
  \caption{Cellular and finite models of $\Kl$.}
  \label{fig_Q10}
\end{figure}

As was mentioned above, another finite model of the torus will be found in the proof of \ref{theo_torus_Klein}: the space $\T_{1,1}$ of figure \ref{fig_finite_Q11}. Observe that $\T_{1,1}$ is the face poset of the regular CW-complex structure of $T^{2}$ given in figure \ref{fig_CW_Q11}

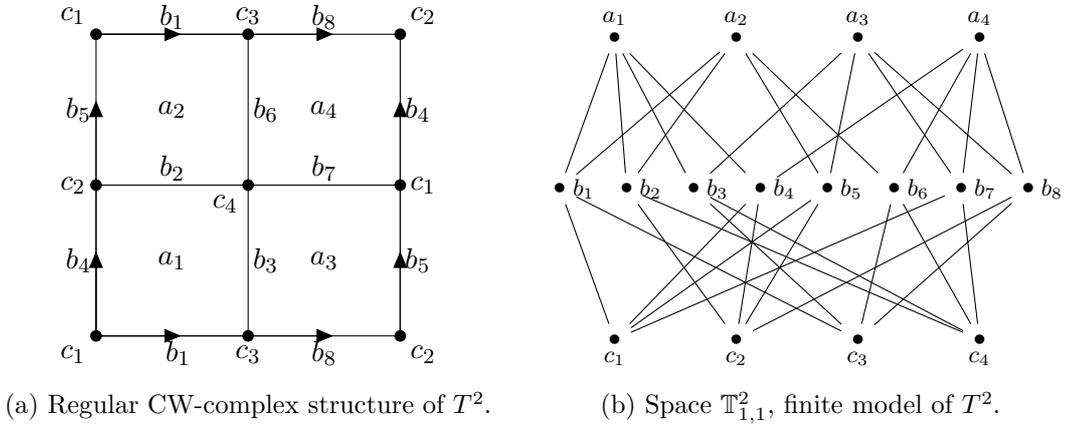
\begin{figure}[h!]
\label{fig_Q11}
  \begin{subfigure}[t]{0.45\textwidth}
  \begin{center}
  \begin{tikzpicture}[x=4cm,y=4cm]
   \tikzstyle{every node}=[font=\normalsize]

  \draw[color=black]
    (0,0) node[below left]{$c_1$} --
	(0.5,0) node[below]{$c_3$}--
	(1,0) node[below right]{$c_2$};
  \draw[color=black]
    (0,0.5) node[left]{$c_2$} --
	(0.5,0.5) node[below left]{$c_4$}--
	(1,0.5) node[right]{$c_1$};
  \draw[color=black]
    (0,1) node[above left]{$c_1$} --
	(0.5,1) node[above]{$c_3$}--
	(1,1) node[above right]{$c_2$};
  \draw 	(0,0) -- (0,1);
  \draw (0.5,0) -- (0.5,1);
  \draw (1,0) -- (1,1);

 \fill[fill=black] (0,0) circle (.2em);
 \fill[fill=black] (0,0.5) circle (.2em);
 \fill[fill=black] (0,1) circle (.2em);
 \fill[fill=black] (0.5,0) circle (.2em);
 \fill[fill=black] (0.5,0.5) circle (.2em);
 \fill[fill=black] (0.5,1) circle (.2em);
 \fill[fill=black] (1,0) circle (.2em);
 \fill[fill=black] (1,0.5) circle (.2em);
 \fill[fill=black] (1,1) circle (.2em);

 \draw[\arr] (0,0) -- (0,0.28); 
 \draw[\arr] (0,0) -- (0,0.78); 
 \draw[\arr] (1,0) -- (1,0.28);
 \draw[\arr] (1,0) -- (1,0.78);

 \draw[\arr] (0,0) -- (0.28,0);
 \draw[\arr] (0,0) -- (0.78,0);
 \draw[\arr] (0,1) -- (0.28,1);
 \draw[\arr] (0,1) -- (0.78,1);

 \draw(0,0.25) node[left=-2] {$b_4$};
 \draw(0,0.75) node[left=-2] {$b_5$};
 \draw(1,0.25) node[right=-2] {$b_5$};
 \draw(1,0.75) node[right=-2] {$b_4$};

 \draw(0.27,0) node[below=-2] {$b_1$};
 \draw(0.75,0) node[below=-2] {$b_8$};
 \draw(0.25,1) node[above=-2] {$b_1$};
 \draw(0.75,1) node[above=-2] {$b_8$};

\draw (0.5,0.25) node[right=-2]{$b_3$};
\draw (0.5,0.75) node[right=-2]{$b_6$};
\draw (0.25,0.5) node[above=-2]{$b_2$};
\draw (0.75,0.5) node[above=-2]{$b_7$};
\draw (0.25,0.25) node{$a_1$};
\draw (0.25,0.75) node{$a_2$};
\draw (0.75,0.25) node{$a_3$};
\draw (0.75,0.75) node{$a_4$};

\end{tikzpicture}
    \caption{Regular CW-complex structure of $T^{2}$.}
    \label{fig_CW_Q11}
      \end{center}
  \end{subfigure}
  \quad
  \begin{subfigure}[t]{0.45\textwidth}
\begin{tikzpicture}[x=4cm,y=4cm]
   \tikzstyle{every node}=[font=\footnotesize]

\foreach \x in {1,...,4} \draw (0.4*\x,1) node(a\x){$\bullet$} node[above=1]{$a_{\x}$};
\foreach \x in {1,...,8} \draw (0.22*\x,0.5) node(b\x){$\bullet$} node[right=1]{$b_{\x}$};
\foreach \x in {1,...,4} \draw (0.4*\x,0) node(c\x){$\bullet$} node[below=1]{$c_{\x}$};

\foreach \x in {1,2,3,4} \draw (a1)--(b\x);
\foreach \x in {1,2,5,6} \draw (a2)--(b\x);
\foreach \x in {3,5,7,8} \draw (a3)--(b\x);
\foreach \x in {4,6,7,8} \draw (a4)--(b\x);

\foreach \x in {1,4,5,7} \draw (c1)--(b\x);
\foreach \x in {2,4,5,8} \draw (c2)--(b\x);
\foreach \x in {1,3,6,8} \draw (c3)--(b\x);
\foreach \x in {2,3,6,7} \draw (c4)--(b\x);

\end{tikzpicture}
  \caption{Space $\T_{1,1}$, finite model of $T^{2}$.}
  \label{fig_finite_Q11}
  \end{subfigure}
  \caption{Cellular and finite models of $T^{2}$.}

\end{figure}

On the other hand, it is clear that $\Kl_{1,0}^\op$ is another finite model of the Klein bottle. It is easy to check that $\Kl_{1,0}^\op$ is homeomorphic to the space $\Kl_{0,1}$ of figure \ref{fig_finite_Q01}, which is the face poset of the regular CW-complex structure of $\Kl$ given in figure \ref{fig_CW_Q01}.

\begin{figure}[h!]
  \begin{subfigure}[t]{0.45\textwidth}
  \begin{center}
\begin{tikzpicture}[x=4cm,y=4cm]
   \tikzstyle{every node}=[font=\normalsize]

   \draw[color=black]
    (0,0) node[below left]{$c_1$} --
	(0.5,0) node[below]{$c_3$}--
	(1,0) node[below right]{$c_2$};
  \draw[color=black]
    (0,0.5) node[left]{$c_2$} --
	(0.5,0.5) node[below left]{$c_4$}--
	(1,0.5) node[right]{$c_1$};
  \draw[color=black]
    (0,1) node[above left]{$c_1$} --
	(0.5,1) node[above]{$c_3$}--
	(1,1) node[above right]{$c_2$};
  \draw 	(0,0) -- (0,1);
  \draw (0.5,0) -- (0.5,1);
  \draw (1,0) -- (1,1);

 \draw(0,0.25) node[left=-1] {$b_1$};
 \draw(0,0.75) node[left=-1] {$b_7$};
 \draw(1,0.25) node[right=-1] {$b_1$};
 \draw(1,0.75) node[right=-1] {$b_7$};

 \draw(0.27,0) node[below=-1] {$b_3$};
 \draw(0.75,0) node[below=-1] {$b_6$};
 \draw(0.25,1) node[above=-1] {$b_3$};
 \draw(0.75,1) node[above=-1] {$b_6$};

\draw (0.5,0.25) node[right=-2]{$b_2$};
\draw (0.5,0.75) node[right=-2]{$b_8$};
\draw (0.25,0.5) node[above=-2]{$b_4$};
\draw (0.75,0.5) node[above=-2]{$b_5$};
\draw (0.25,0.25) node{$a_1$};
\draw (0.25,0.75) node{$a_3$};
\draw (0.75,0.25) node{$a_2$};
\draw (0.75,0.75) node{$a_4$};

 \fill[fill=black] (0,0) circle (.2em);
 \fill[fill=black] (0,0.5) circle (.2em);
 \fill[fill=black] (0,1) circle (.2em);
 \fill[fill=black] (0.5,0) circle (.2em);
 \fill[fill=black] (0.5,0.5) circle (.2em);
 \fill[fill=black] (0.5,1) circle (.2em);
 \fill[fill=black] (1,0) circle (.2em);
 \fill[fill=black] (1,0.5) circle (.2em);
 \fill[fill=black] (1,1) circle (.2em);

 \draw[\arr] (0,0) -- (0,0.28); 
 \draw[\arr] (0,0) -- (0,0.78); 
 \draw[\arr] (1,0.28)--(1,0.22);
 \draw[\arr] (1,0.78) -- (1,0.72);

 \draw[\arr] (0,0) -- (0.28,0);
 \draw[\arr] (0,0) -- (0.78,0);
 \draw[\arr] (0,1) -- (0.28,1);
 \draw[\arr] (0,1) -- (0.78,1);
\end{tikzpicture}
    \caption{Regular CW-complex structure of $\Kl$.}
    \label{fig_CW_Q01}
      \end{center}
  \end{subfigure}
  \quad
  \begin{subfigure}[t]{0.45\textwidth}
\begin{tikzpicture}[x=4cm,y=4cm]
   \tikzstyle{every node}=[font=\footnotesize]

\foreach \x in {1,...,4} \draw (0.4*\x,1) node(a\x){$\bullet$} node[above=1]{$a_{\x}$};
\foreach \x in {1,...,8} \draw (0.22*\x,0.5) node(b\x){$\bullet$} node[right=1]{$b_{\x}$};
\foreach \x in {1,...,4} \draw (0.4*\x,0) node(c\x){$\bullet$} node[below=1]{$c_{\x}$};

\foreach \x in {1,2,3,4} \draw (a1)--(b\x);
\foreach \x in {1,2,5,6} \draw (a2)--(b\x);
\foreach \x in {3,4,7,8} \draw (a3)--(b\x);
\foreach \x in {5,6,7,8} \draw (a4)--(b\x);

\foreach \x in {1,3,5,7} \draw (c1)--(b\x);
\foreach \x in {1,4,6,7} \draw (c2)--(b\x);
\foreach \x in {2,3,6,8} \draw (c3)--(b\x);
\foreach \x in {2,4,5,8} \draw (c4)--(b\x);

\end{tikzpicture}
  \caption{Space $\Kl_{0,1}$, finite model of $\Kl$.}
  \label{fig_finite_Q01}
  \end{subfigure}
  \caption{Cellular and finite models of $\Kl$.}
  \label{fig_Q01}
\end{figure}

We will give now two lemmas that will be used several times in this section.

\begin{lemma} \label{lemma_relp_no_S2}
Let $X$ be a finite and connected T$_0$--space without beat points. If $X$ does not satisfy condition \Spl[2], then for all $(x,y)\in \M$ either $\#\Relp^{-1} (x,y)\geq 2$ or $\Relp^{-1} (x,y)=\{b\}$ and $\min\{\alpha_b,\beta_b\}\geq 3$.
\end{lemma}

\begin{proof}
As in the proof of \ref{lemma_relp_no_S1} we obtain that $\Relp^{-1}(x,y)\neq\varnothing$ for all $(x,y)\in \M$. Suppose that there exist $(x,y)\in\M$ and $b\in\B$ such that $\Relp^{-1} (x,y)=\{b\}$ and $\alpha_b=2$.

By \ref{prop_non_Hausdorff_suspension} $F_b\cup F_y$ is homotopy equivalent to $\Su(F_b\cap F_y)$. If $F_b\cap F_y$ is connected then the connected components of $F_b\cup F_y \cup \mxl(X)$ have trivial fundamental group and hence $(X;V_x,F_b\cup F_y \cup \mxl(X))$ is a splitting triad that satisfies \Spl[2], which entails a contradiction.

Thus, $F_b\cap F_y$ is not connected. Note that $\mxl(F_b\cap F_y)\subseteq \mxl(F_b) = F_b \cap \mxl(X)$. Then $F_b\cap F_y$ has, at most, two maximal points. But since $F_b\cap F_y$ is not connected, $F_b\cap F_y$ has exactly two maximal points each of which must be the maximum of its connected component. Thus $F_b\cap F_y$ is homotopy equivalent to $D_2$ and hence $F_b\cup F_y$ is homotopy equivalent to $\Su D_2$. Then, the splitting triad $(X;F_b\cup F_y \cup \mxl(X),V_x)$ satisfies \Spl[2].
\end{proof}

\begin{lemma} \label{lemma_relp_card_B_no_S2}
Let $X$ be a finite and connected T$_0$--space without beat points. Let $\B'=\{b\in \B\tq \min\{\alpha_b,\beta_b\}\geq 3\}$ and let $\B''=\B-\B'$.

Suppose that $X$ does not satisfy condition \Spl[2]. Then, with the notations above, 
$$2\#\B'(m_X-3)(n_X-3)+\#\B''(m_X-2)(n_X-2)\geq 2m_Xn_X$$
and the equality holds only if $\alpha_b=\beta_b=2$ for all $b\in\B''$.

In addition, if the equality holds and $\min\{m_X,n_X\}\geq 4$ then $\alpha_b=\beta_b=3$ for all $b\in\B'$.
\end{lemma}

\begin{proof}
Note that $\#\mxl(X)\geq 3$ and $\#\mnl(X)\geq 3$ by \ref{rem_mxl_and_cond}.

Let $\M'=\{(x,y)\in\M\tq \#\Relp^{-1}(x,y)=1\}$ and let $\M''=\M-\M'$. By \ref{lemma_relp_no_S2}, 
$$\#\Relp\geq \#\M'+2\#\M'' =2mn-\#\M' \geq 2mn - \sum\limits_{b\in \B'}\#\Relp(b) = 2mn-\sum\limits_{b\in \B'}(m-\alpha_b)(n-\beta_b)$$

On the other hand, $\#\Relp=\sum\limits_{b\in \B'}(m-\alpha_b)(n-\beta_b)+\sum\limits_{b\in \B''}(m-\alpha_b)(n-\beta_b)$. Thus, 

$$2mn\leq 2\sum\limits_{b\in \B'}(m-\alpha_b)(n-\beta_b)+\sum\limits_{b\in \B''}(m-\alpha_b)(n-\beta_b) \leq 2\#\B'(m-3)(n-3)+\B''(m-2)(n-2)\ .$$

The result follows.
\end{proof}

The following theorem is another major result of this article.

\begin{theo} \label{theo_16_points}
Let $X$ be a finite and connected T$_0$--space without beat points and which does not satisfy condition \Spl[2]. Then $\# X\geq 16$.
\end{theo}

\begin{proof}
As before, $\#\mxl(X)\geq 3$ and $\#\mnl(X)\geq 3$ by \ref{rem_mxl_and_cond}. Now, by \ref{lemma_relp_card_B_no_S2}, 
$$2mn\leq 2\#\B'(m-3)(n-3)+\#\B''(m-2)(n-2)\leq \#\B . \max\{2(m-3)(n-3),(m-2)(n-2)\}.$$

If $\max\{2(m-3)(n-3),(m-2)(n-2)\}=2(m-3)(n-3)$, let $a=m-3$ and $b=n-3$. Then
\begin{displaymath}
\begin{array}{rcl}
\#X & = & \displaystyle m+n+\#\B \geq m+n+ \frac{mn}{(m-3)(n-3)} = a+b+\frac{3}{a}+\frac{3}{b}+\frac{9}{ab}+7= \\
& = & \displaystyle \left(\frac{a}{3}+\frac{3}{a}\right)+\left(\frac{b}{3}+\frac{3}{b}\right)+\left(\frac{9}{ab}+\frac{2a}{3}+\frac{2b}{3}\right)+7 \geq 2+2+3\sqrt[3]{4}+7 > 15
\end{array}
\end{displaymath}
where the penultimate inequality holds by the inequality of arithmetic and geometric means.

If $\max\{2(m-3)(n-3),(m-2)(n-2)\}=(m-2)(n-2)$, let $a=m-2$ and $b=n-2$. Then
\begin{displaymath}
\begin{array}{rcl}
\#X & = & \displaystyle m+n+\#\B \geq m+n+ \frac{2mn}{(m-2)(n-2)} = a+b+\frac{4}{a}+\frac{4}{b}+\frac{8}{ab}+6= \\
& = & \displaystyle \left(\frac{a}{2}+\frac{4}{a}\right)+\left(\frac{b}{2}+\frac{4}{b}\right)+\left(\frac{8}{ab}+\frac{a}{2}+\frac{b}{2}\right)+6 \geq 2\sqrt{2}+2\sqrt{2}+3\sqrt[3]{2}+6 > 15
\end{array}
\end{displaymath}
where, as before, the penultimate inequality holds by the inequality of arithmetic and geometric means.

Therefore, $\#X\geq 16$.
\end{proof}

As a corollary of the previous theorem we obtain an affirmative answer to the question posed in \cite[p.44]{BarLN} about the minimality of $\Su S^0 \times \Su S^0$ as a finite model of the torus.

\begin{coro} \label{coro_finite_model_of_torus}
Let $X$ be a finite T$_0$--space. If $X$ is a finite model of the torus then $\# X\geq 16$.
\end{coro}

\begin{proof}
Let $X$ be a finite model of the torus. Without loss of generality we may suppose that $X$ does not have beat points. From \ref{coro_torus_Klein_bottle_not_S2} we obtain that $X$ does not satisfy \Spl[2]. Therefore, $\# X\geq 16$ by \ref{theo_16_points}.
\end{proof}

\begin{coro} \label{coro_finite_model_of_Klein_bottle}
Let $X$ be a finite T$_0$--space. If $X$ is a finite model of the Klein bottle then $\# X\geq 16$.
\end{coro}

\begin{proof}
By \ref{coro_torus_Klein_bottle_not_S2}, $X$ does not satisfy \Spl[2]. The result follows from \ref{theo_16_points}.
\end{proof}

Now we will find all the minimal finite models of the torus and the Klein bottle.

Let $X$ be a finite and connected T$_0$--space without beat points which does not satisfy \Spl[2]. For each $a\in\mxl(X)$ we define 
$$\sigma_a=\sum_{b\in\Relm^{-1}(a)}\frac{1}{\#\Relm(b)}.$$

\begin{rem} \label{rem_relm_geq_2}
Note that from \ref{lemma_relp_no_S2} it follows that if $X$ is a finite and connected T$_0$--space without beat points which does not satisfy \Spl[2] then $\#\Relm^{-1}(a)\geq 2$ for all $a\in\mxl(X)$.
\end{rem}

\begin{lemma} \label{lemma_card_B-Ua=2}
Let $X$ be a finite and connected T$_0$--space without beat points which does not satisfy \Spl[2]. Let $a\in\mxl(X)$. If $\Relm^{-1}(a)=\{b_1,b_2\}$ (with $b_1\neq b_2$) then $\{b_1,b_2\}$ is an antichain, $\#(F_{b_1}\cap F_{b_2}\cap\mxl(X))\geq 3$, $U_{b_1}\cap U_{b_2}=\varnothing$ and $\min\{\beta_{b_1},\beta_{b_2}\}\geq 3$.
\end{lemma}

\begin{proof}
Suppose that $\Relm^{-1}(a)=\{b_1,b_2\}$. If either $F_{b_1}\cap F_{b_2}=\varnothing$ or $b_1\leq b_2$ or $b_2\leq b_1$ then the connected components of $X-V_a^X$ are contractible and hence the splitting triad $(X;X-V_a^X,V_a^X)$ satisfies \Spl[2]. Thus, $F_{b_1}\cap F_{b_2}\neq\varnothing$ and $\{b_1,b_2\}$ is an antichain. Hence, the connected components of $X-V_a^X$ are $F_{b_1}\cup F_{b_2}$ or singletons.

Note that $F_{b_1}\cap F_{b_2}\subseteq \mxl(X)$. Hence $F_{b_1}\cap F_{b_2}$ is discrete. And since $F_{b_1}\cup F_{b_2}$ is homotopy equivalent to $\Su (F_{b_1}\cap F_{b_2})$ (by \ref{prop_non_Hausdorff_suspension}) we obtain that if $\#(F_{b_1}\cap F_{b_2})\leq 2$ then the splitting triad $(X;X-V_a^X,V_a^X)$ satisfies \Spl[2]. Therefore, $\#(F_{b_1}\cap F_{b_2}\cap\mxl(X))\geq 3$.

Now, if $w\in U_{b_1}\cap U_{b_2}$ then $\Relp^{-1}(a,w)=\varnothing$ contradicting \ref{lemma_relp_no_S2}. Hence, $U_{b_1}\cap U_{b_2}=\varnothing$. Let $w_1\in U_{b_1}\cap \mnl(X)$. Then $b_2 \notin F_{w_1}$ and hence $\Relp^{-1}(a,w_1)=\{b_2\}$. Thus, $\beta_{b_2}\geq 3$ by \ref{lemma_relp_no_S2}. In a similar way we obtain that $\beta_{b_1}\geq 3$.
\end{proof}

\begin{lemma} \label{lemma_sigma_a}
Let $X$ be a finite and connected T$_0$--space without beat points which does not satisfy \Spl[2] and such that $m_X\geq 4$. Let $a\in\mxl(X)$. Then $\sigma_a\geq\min\left\{\frac{2}{m_X-3},\frac{3}{m_X-2}\right\}$.

If, in addition, $n_X\leq 5$ then $\sigma_a\geq \frac{2}{m_X-3}$.
\end{lemma}

\begin{proof}
By \ref{rem_relm_geq_2}, $\#\Relm^{-1}(a)\geq 2$. If $\#\Relm^{-1}(a)\geq 3$ then 
$$\sigma_a=\sum_{b\in\Relm^{-1}(a)}\frac{1}{\#\Relm(b)}\geq \frac{3}{m_X-2}$$
by \ref{rem_minimal_space_two_points}. Suppose, in addition, that $n_X\leq 5$. If $\#\Relm^{-1}(a)\geq 4$ then $\sigma_a\geq \frac{4}{m_X-2} \geq \frac{2}{m_X-3}$ since $m_X\geq 4$. Hence, we may assume that $\#\Relm^{-1}(a)=3$. Since $n_X\leq 5$, from \ref{rem_minimal_space_two_points} we obtain that there exist distinct points $b_1,b_2\in \Relm^{-1}(a)$ and $y\in\mnl(X)$ such that $b_1,b_2\in F_y$. Let $b_3 \in \Relm^{-1}(a)-\{b_1,b_2\}$. From \ref{lemma_relp_no_S2} it follows that $\Relp^{-1}(a,y)=\{b_3\}$ and $\alpha_{b_3}\geq 3$. Thus, 
$$\sigma_a=\sum_{b\in\Relm^{-1}(a)}\frac{1}{\#\Relm(b)}\geq \frac{2}{m_X-2}+\frac{1}{m_X-3}\geq \frac{2}{m_X-3}$$
since $m_X\geq 4$.

Suppose now that $\#\Relm^{-1}(a)=2$ and let $b_1,b_2\in \B$ such that $\Relm^{-1}(a)=\{b_1,b_2\}$. Thus, $\#(F_{b_1}\cap F_{b_2})\geq 3$ by \ref{lemma_card_B-Ua=2} and hence $\alpha_{b_j}\geq 3$ for $j\in\{1,2\}$. Thus,
$$\sigma_a=\sum_{b\in\Relm^{-1}(a)}\frac{1}{\#\Relm(b)}\geq \frac{2}{m_X-3}.$$
\end{proof}

\begin{lemma} \label{lemma_relm_card_B_no_S2}
Let $X$ be a finite and connected T$_0$--space without beat points which does not satisfy condition \Spl[2] and such that $m_X\geq 4$. Then 
$$\# \B \geq m_X . \min\left\{\tfrac{2}{m_X-3}\;,\;\tfrac{3}{m_X-2}\right\}= \left\{ 
\begin{array}{cr}
\frac{2m_X}{m_X-3} & \textnormal{if $m_X\geq 5$} \\ 
\frac{3m_X}{m_X-2} & \textnormal{if $m_X\leq 5$} 
\end{array}
\right. \ .$$
If, in addition, $n_X\leq 5$ then $\#\B \geq \frac{2m_X}{m_X-3}$.
\end{lemma}

\begin{proof} 
Note that
\begin{displaymath}
\begin{array}{rcl}
\#\B & \geq & \displaystyle \#\Relm^{-1}(\mxl(X)) = \sum_{b\in \Relm^{-1}(\mxl(X))}\left(\sum_{a\in \Relm(b)} \frac{1}{\# \Relm(b)}\right)= \sum_{(a,b)\in \Relm} \frac{1}{\# \Relm(b)} = \\ & = & \displaystyle \sum\limits_{a\in\mxl(X)}\left( \sum\limits_{b\in \Relm^{-1}(a)} \frac{1}{\# \Relm(b)} \right) = \sum\limits_{a\in\mxl(X)} \sigma_a \ .
\end{array}
\end{displaymath}
The result now follows from \ref{lemma_sigma_a}.
\end{proof}

The following theorem describes all the minimal finite models of the torus and the Klein bottle.

\begin{theo}\label{theo_torus_Klein}
Let $X$ be a finite T$_0$--space such that $\#X=16$. 
\begin{enumerate}
\item If $X$ is a finite model of the torus then $X$ is homeomorphic to either $\T_{0,0}$ or $\T_{1,1}$.
\item If $X$ is a finite model of the Klein bottle then $X$ is homeomorphic to either $\Kl_{1,0}$ or $\Kl_{0,1}$.
\end{enumerate}
\end{theo}

\begin{proof}
Suppose that $X$ is either a finite model of the torus or a finite model of the Klein bottle. Hence, by \ref{coro_torus_Klein_bottle_not_S2}, $X$ does not satisfy \Spl[2]. Thus, 
by \ref{rem_mxl_and_cond}, $\# \mxl (X) \geq 3$. Also, $\mxl(X)\cap\mnl(X)=\varnothing$ since $X$ is connected. Without loss of generality, we may assume that $\#\mxl(X)\leq\#\mnl(X)$. From \ref{coro_finite_model_of_torus} and \ref{coro_finite_model_of_Klein_bottle} we obtain that $X$ does not have beat points.

\underline{Case 1}: $\# \mxl(X) = 3$. By \ref{lemma_relp_card_B_no_S2}, $\#\B\geq \frac{6n}{n-2}$. Thus,
\begin{displaymath}
\#X=3+n+\#\B\geq 3+n+ \frac{6n}{n-2} = 11+(n-2)+\frac{12}{n-2}\geq 11 + 2\sqrt{12}>17
\end{displaymath}
where the penultimate inequality holds by the inequality of arithmetic and geometric means.

\underline{Case 2}: $\# \mxl(X) = 4$.

\underline{Case 2.1}: $\# \mnl(X) = 4$ or $\# \mnl(X) = 5$. From \ref{lemma_relm_card_B_no_S2} we obtain that $\# \B \geq 8$. Hence, $\#\B=8$ and $\#\mnl(X)=4$.

From \ref{lemma_relp_card_B_no_S2} we obtain that $2\#\B'+4\#\B''\geq 32$. Hence, $\#\B''=8$, $\#\B'=0$ and the equality holds. Thus,  $\B''=\B$ and from \ref{lemma_relp_no_S2} and the proof of \ref{lemma_relp_card_B_no_S2} it follows that $\alpha_b=\beta_b=2$ for all $b\in\B$ and $\#\Relp^{-1}(x,y)=2$ for all $(x,y)\in\M$. We will prove now that $\#(U_a\cap\B)\leq 4$ for all $a\in\mxl(X)$. Let $a\in\mxl(X)$. Let $y_1\in\mnl(X)$ and let $b_1$ and $b_2$ be the elements of $\Relp^{-1}(a,y_1)$. Since $y_1\notin U_{b_1}\cup U_{b_2}$ and $\beta_{b_1}=\beta_{b_2}=2$ it follows that there exists $y_2\in\mnl(X)$ such that $y_2\in U_{b_1}\cap U_{b_2}$. Since $\#\Relp^{-1}(a,y_2)=2$, it follows that $\#\Relm^{-1}(a)\geq 4$ and hence $\#(U_a\cap\B)\leq 4$.

Since $\alpha_b=2$ for all $b\in\B$ and $\#(U_a\cap\B)\leq 4$ for all $a\in\mxl(X)$ we obtain that $\#(U_a\cap\B)= 4$ for all $a\in\mxl(X)$. In a similar way we obtain that $\#(F_y\cap\B)= 4$ for all $y\in\mnl(X)$. And since $\#\Relp^{-1}(x,y)=2$ for all $(x,y)\in\M$ it follows that $\#(U_x\cap F_y\cap\B)=2$ for all $(x,y)\in\M$. In particular, $x>y$ for all $(x,y)\in\M$.

We claim that $\B$ is an antichain. Indeed, suppose that there exist $b_1,b_2\in \B$ with $b_1<b_2$. Since $\alpha_{b_1}=\alpha_{b_2}=2$ we obtain that $F_{b_1}\cap \mxl(X)= F_{b_2}\cap \mxl(X)$. Since $b_1$ is not an up beat point of $X$, there exists $b_3\in \widehat F_{b_1} \cap \B$ such that $b_3\neq b_2$. As before, we obtain that $F_{b_3}\cap \mxl(X)= F_{b_1}\cap \mxl(X)$. Let $x_0\in F_{b_1}\cap \mxl(X)$ and let $y_0\in U_{b_1}\cap \mnl(X)$. Then, $\{b_1,b_2,b_3\}\subseteq U_{x_0}\cap F_{y_0}\cap\B$ which entails a contradiction.

We will prove now that $\#(U_{a_1}\cap U_{a_2}\cap \B)\leq 3$ for all $a_1,a_2\in\mxl(X)$ with $a_1\neq a_2$. Let $a_1,a_2\in\mxl(X)$ with $a_1\neq a_2$ and suppose that $\#(U_{a_1}\cap U_{a_2}\cap \B)=4$. Thus $U_{a_1}\cap \B = U_{a_2}\cap \B$. Let $a_3$ and $a_4$ be the remaining maximal points of $X$. It follows that $U_{a_3}\cap \B = U_{a_4}\cap \B = \B- U_{a_1}$ since $\alpha_b=2$ for all $b\in\B$. Note that $\#(F_y\cap U_{a_j}\cap \B )= 2$ for all $y\in\mnl(X)$ and for all $j\in\{1,3\}$. Let $w\in\mnl(X)$ and let $A=F_w\cup\B$. Since $\widetilde H_0(A)=0$ there is an epimorphism $H_1(X)\to H_1(X,A)$. From proposition 3.2 of \cite{CO} we obtain that 
$H_1(X,A)\cong \bigoplus\limits_{z\in X-A}\widetilde H_0 (\widehat F_z) \cong \Z^3$. Thus, $\rk(H_1(X))\geq 3$ which can not be possible since $X$ is a finite model of either the torus or the Klein bottle. Therefore, $\#(U_{a_1}\cap U_{a_2}\cap \B)\leq 3$ for all $a_1,a_2\in\mxl(X)$ with $a_1\neq a_2$.

Applying this argument to $X^\op$ we obtain that $\#(F_{y_1}\cap F_{y_2}\cap \B)\leq 3$ for all $y_1,y_2\in\mnl(X)$ with $y_1\neq y_2$.

We will prove now that indeed $\#(U_{a_1}\cap U_{a_2}\cap \B)\leq 2$ for all $a_1,a_2\in\mxl(X)$ with $a_1\neq a_2$. Let $a_1$, $a_2$, $a_3$ and $a_4$ be the maximal points of $X$ and suppose that $\#(U_{a_1}\cap U_{a_2}\cap \B)=3$. Let $b_1$, $b_2$ and $b_3$ be the elements of $U_{a_1}\cap U_{a_2}\cap \B$, let $b_4$ be the only element of $U_{a_1}- U_{a_2}$ and let $b_5$ be the only element of $U_{a_2}- U_{a_1}$. Let $b_6$, $b_7$ and $b_8$ be the remaining elements of $\B$. Clearly, $\{b_6,b_7,b_8\}\subseteq U_{a_3}\cap U_{a_4}$. Without loss of generality we may assume that $U_{a_3}\cap\B=\{b_4,b_6,b_7,b_8\}$ and $U_{a_4}\cap\B=\{b_5,b_6,b_7,b_8\}$. Let $c_2$ and $c_3$ be the elements of $U_{b_4}\cap\mnl(X)$. Then, for each $j\in\{2,3\}$ we have that $\#(F_{c_j}\cap \{b_1,b_2,b_3\})=1$ and $\#(F_{c_j}\cap \{b_6,b_7,b_8\})=1$ since $\#(U_{a_1}\cap F_{c_j}\cap\B)=2=\#(U_{a_3}\cap F_{c_j}\cap\B)$. And since $\#(F_{c_j}\cap \B)=4$ for $j\in\{2,3\}$ it follows that $c_j<b_5$ for $j\in\{2,3\}$. Thus $U_{b_4}\cap\mnl(X)=U_{b_5}\cap\mnl(X)=\{c_2,c_3\}$.

Let $c_1$ and $c_4$ be the remaining minimal points of $X$. Note that $\#(F_{c_j}\cap\{b_1,b_2,b_3\})=2$ for $j\in\{1,4\}$. Let $W=\{b_1,b_2,b_3,c_1,c_2,c_3,c_4\}$. Observe that $W=U_{a_1}\cap U_{a_2}$ and that $W$ is homotopy equivalent to $\{b_1,b_2,b_3,c_1,c_4\}$. Note also that $1\leq \#(F_{c_1}^W \cap F_{c_4}^W) \leq 2$.
Suppose that $\#(F_{c_1}^W \cap F_{c_4}^W)=1$. Then $W$ is contractible and hence $U_{a_1}\cup U_{a_2}$ is contractible by \ref{prop_non_Hausdorff_suspension}. Thus, from \ref{prop_weak_collapse} it follows that there exists a finite model of either the torus or the Klein bottle with fewer than 16 points, which contradicts either \ref{coro_finite_model_of_torus} or \ref{coro_finite_model_of_Klein_bottle}. Therefore, $\#(F_{c_1}^W \cap F_{c_4}^W)=2$.

In a similar way, it follows that $\#(F_{c_1} \cap F_{c_4}\cap\{b_6,b_7,b_8\} )=2$. Hence, $\#(F_{c_1} \cap F_{c_4}\cap\B)=4$ which entails a contradiction.

Therefore, $\#(U_{x_1}\cap U_{x_2}\cap \B)\leq 2$ for all $x_1,x_2\in\mxl(X)$ with $x_1\neq x_2$. And applying this argument to $X^\op$ we obtain that $\#(F_{y_1}\cap F_{y_2}\cap \B)\leq 2$ for all $y_1,y_2\in\mnl(X)$ with $y_1\neq y_2$.

Since there are six pairs of maximal points and eight points in $\B$ and $\alpha_b=2$ for all $b\in\B$ we obtain that there exist distinct points $a_1,a_2\in\mxl(X)$ such that $\#(U_{a_1}\cap U_{a_2}\cap\B)=2$. Let $a_3$ and $a_4$ be the remaining maximal points of $X$. Observe that $U_{a_3}\cap U_{a_4}\cap\B=\B-U_{a_1}\cup U_{a_2}$ and hence $\#(U_{a_3}\cap U_{a_4}\cap\B)=2$.

Let $b_1$ and $b_2$ be the elements of $U_{a_1}\cap U_{a_2}\cap\B$, let $b_3$ and $b_4$ be the elements of $\widehat U_{a_1}- U_{a_2}$, let $b_5$ and $b_6$ be the elements of $\widehat U_{a_2}- U_{a_1}$ and let $b_7$ and $b_8$ be the elements of $U_{a_3}\cap U_{a_4}\cap\B$. Since $\alpha_{b_3}=\alpha_{b_4}=2$, without loss of generality we may assume that one of the following two cases hold:
\begin{itemize}
\item $b_3<a_3$ and $b_4<a_3$
\item $b_3<a_3$ and $b_4<a_4$
\end{itemize}
In the first case, we obtain that $b_5<a_4$ and $b_6<a_4$ and thus $X-\mnl(X)$ is homeomorphic to the following space
which will be called $\Q0$.
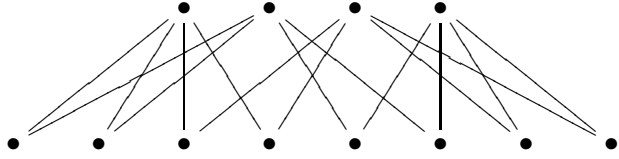
\[
\xymatrix@R=40pt@C=20pt{& & \bullet\ar@{-}[dll]\ar@{-}[dl]\ar@{-}[d]\ar@{-}[dr] & 
\bullet\ar@{-}[dlll]\ar@{-}[dll]\ar@{-}[dr]\ar@{-}[drr] & 
\bullet\ar@{-}[drrr]\ar@{-}[drr]\ar@{-}[dl]\ar@{-}[dll] & 
\bullet\ar@{-}[drr]\ar@{-}[dr]\ar@{-}[d]\ar@{-}[dl] & & \\
\bullet & \bullet & \bullet & \bullet & 
\bullet & \bullet & \bullet & \bullet
}
\]

In the second case, without loss of generality, we may assume that $b_5<a_3$ and $b_6<a_4$ and thus $X-\mnl(X)$ is homeomorphic to the following space which will be called $\Q1$.
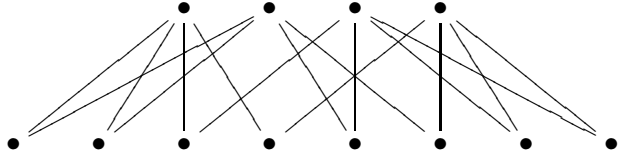
\[
\xymatrix@R=40pt@C=20pt{& & \bullet\ar@{-}[dll]\ar@{-}[dl]\ar@{-}[d]\ar@{-}[dr] & 
\bullet\ar@{-}[dlll]\ar@{-}[dll]\ar@{-}[dr]\ar@{-}[drr] & 
\bullet\ar@{-}[drrr]\ar@{-}[drr]\ar@{-}[d]\ar@{-}[dll] & 
\bullet\ar@{-}[drr]\ar@{-}[dr]\ar@{-}[d]\ar@{-}[dll] & & \\
\bullet & \bullet & \bullet & \bullet & 
\bullet & \bullet & \bullet & \bullet
}
\]
Now, we claim that in any of the cases above, since $U_{a_1}\cap U_{a_2}\cap \B = \{b_1,b_2\}$ then $U_{b_1}\cap U_{b_2}=\varnothing$. Indeed, suppose that $z\in U_{b_1}\cap U_{b_2}$. Since $\#(U_{a_j}\cap F_{z}\cap\B)=2$ for all $j\in\{1,2,3,4\}$ we obtain that $z<b_7$ and $z<b_8$. Thus, if $z_1$ and $z_2$ are distinct elements of $U_{b_1}\cap U_{b_2}$ then $\#(F_{z_1}\cap F_{z_2}\cap\B)=4$ which entails a contradiction. Therefore, $\#(U_{b_1}\cap U_{b_2})\leq 1$. If $U_{b_1}\cap U_{b_2}= \{z_1\}$ then $U_{a_1}\cap U_{a_2}=\{b_1,b_2\}\cup \mnl(X)$ which is the disjoint union of two contractible spaces. Hence, by \ref{prop_non_Hausdorff_suspension}, $U_{a_1}\cup U_{a_2}$ is homotopy equivalent to $\Su D_2$. Note that the inclusion of $X-(U_{a_1}\cup U_{a_2})=\{a_3,a_4,b_7,b_8\}$ into $X$ induces the trivial map between the fundamental groups since $X-(U_{a_1}\cup U_{a_2})$ is included in the contractible subspace $\{a_3,a_4,b_7,b_8,z_1\}$. Therefore, $U_{b_1}\cap U_{b_2}=\varnothing$ as desired.

Clearly, the same conclusion holds for the pair $b_7$, $b_8$ in both cases and for the pairs $b_3$, $b_4$ and $b_5$, $b_6$ in the first case. Applying the same reasoning to $X^\op$ it follows that if $y_1,y_2\in\mnl(X)$ are such that $F_{y_1}\cap F_{y_2}\cap \B = \{b',b''\}$ (with $b'\neq b''$) then $F_{b'}\cap F_{b''}=\varnothing$.

As above, let $\mnl(X)=\{c_1,c_2,c_3,c_4\}$.

We will analyze now the first case, that is, $X-\mnl(X)$ is homeomorphic to $\Q0$. Without loss of generality we may assume that $\#(F_{c_1}\cap F_{c_2}\cap \B)=2$ and that $b_1\in F_{c_1}\cap F_{c_2}$. Let $j\in\{2,3,\ldots,8\}$ be such that $F_{c_1}\cap F_{c_2}\cap \B=\{b_1,b_j\}$. Thus, $F_{b_1}\cap F_{b_j}=\varnothing$ and hence $j=7$ or $j=8$. Without loss of generality we may suppose that $j=7$. Now, since $U_{b_1}\cap U_{b_2}=\varnothing$ and $U_{b_7}\cap U_{b_8}=\varnothing$ we obtain that $U_{b_2}\cap\mnl(X)=U_{b_8}\cap\mnl(X)=\{c_3,c_4\}$. And since $U_{b_3}\cap U_{b_4}=\varnothing$ and $U_{b_5}\cap U_{b_6}=\varnothing$ then $\#(F_{c_k}\cap\{b_3,b_4\})=\#(F_{c_k}\cap\{b_5,b_6\})=1$ for all $k\in\{1,2,3,4\}$. Hence, without loss of generality we may assume that $F_{c_1}\cap\{b_3,b_4,b_5,b_6\}=\{b_3,b_5\}$. And since $F_{c_1}\cap F_{c_2}\cap \B=\{b_1,b_7\}$ it follows that $F_{c_2}\cap\{b_3,b_4,b_5,b_6\}=\{b_4,b_6\}$.

Now, relabelling $c_3$ and $c_4$ if necessary we may suppose that $c_3<b_3$ and since $\#(F_{c_3}\cap F_{c_4}\cap \B)\leq 2$ and $b_2,b_8\in F_{c_3}\cap F_{c_4}\cap \B$ we obtain that $c_4<b_4$.

There are now two possibilities: $c_3<b_5$ or $c_3<b_6$.

If $c_3<b_5$ then $c_4<b_6$ and $X-\mxl(X)$ is homeomorphic to $\Q0^\op$. In this case $X$ is homeomorphic to the space $\T_{0,0}$ of page \pageref{fig_finite_Q00}.

If $c_3<b_6$ then $c_4<b_5$ and $X-\mxl(X)$ is homeomorphic to $\Q1^\op$. In this case $X$ is homeomorphic to the space $\Kl_{0,1}$ of page \pageref{fig_finite_Q01}.

We analyze now the second case, that is, $X-\mnl(X)$ is homeomorphic to $\Q1$. Without loss of generality we may suppose that $U_{b_4}\cap\mnl(X)=\{c_1,c_2\}$. Since $\#(F_{c_1}\cap U_{a_1}\cap\B)=2=\#(F_{c_1}\cap U_{a_4}\cap\B)$ it follows that $\#(F_{c_1}\cap\{b_1,b_2,b_3\})=\#(F_{c_1}\cap\{b_6,b_7,b_8\})=1$ and since $\#(F_{c_1}\cap\B)=4$ we obtain that $b_5\in F_{c_1}$. In a similar way we obtain that $b_5\in F_{c_2}$.

Now, since $U_{b_1}\cap U_{b_2}=\varnothing$ and $\beta_{b_1}=\beta_{b_2}=2$ it follows that $U_{b_1}\cup U_{b_2}\supseteq \mnl(X)$. Therefore, $F_{c_1}\cap\{b_1,b_2\}\neq \varnothing$ and hence $b_3\notin F_{c_1}$. In a similar way, $b_6\notin F_{c_1}$ and $b_3,b_6\notin F_{c_2}$. Thus, $U_{b_3}\cap\mnl(X)=U_{b_6}\cap\mnl(X)=\{c_3,c_4\}$.

Without loss of generality we may assume that $c_1<b_1$ and that $c_1<b_7$. Since $\#(F_{c_1}\cap F_{c_2}\cap \B)\leq 2$ and $b_4,b_5\in F_{c_1}\cap F_{c_2}\cap \B$ we obtain that $F_{c_2}\cap\B=\{b_2,b_4,b_5,b_8\}$. Now, relabelling $c_3$ and $c_4$ if necessary, we may suppose that $U_{b_1}\cap\mnl(X)=\{c_1,c_3\}$ and hence $U_{b_2}\cap\mnl(X)=\{c_2,c_4\}$.

There are now two possibilities: $c_3<b_7$ or $c_3<b_8$.

If $c_3<b_7$ then $c_4<b_8$ and $X-\mxl(X)$ is homeomorphic to $\Q0^\op$. In this case $X$ is homeomorphic to the space $\Kl_{1,0}$ of figure \ref{fig_finite_Q10} of page \pageref{fig_finite_Q10}.

If $c_3<b_8$ then $c_4<b_7$ and $X-\mxl(X)$ is homeomorphic to $\Q1^\op$. In this case $X$ is homeomorphic to the space $\T_{1,1}$ of figure \ref{fig_finite_Q11} of page \pageref{fig_finite_Q10}.

\underline{Case 2.2}: $\# \mnl(X) \geq 6$. In this case, $\#\B \geq 6$ by \ref{lemma_relm_card_B_no_S2}. Hence, $\#\B=6$ and $\#\mnl(X)=6$. Thus, from the proof of \ref{lemma_relm_card_B_no_S2} we obtain that $\Relm^{-1}(\mxl(X))=\B$ and that $\sigma_a=\frac{3}{2}$ for all $a\in\mxl(X)$.

We will prove now that for all $a\in\mxl(X)$ we have that $\#\Relm^{-1}(a)=3$ and that the subsets $U_b$, $b\in\Relm^{-1}(a)$, are pairwise disjoint. In particular, we will obtain that $\Relm^{-1}(a)$ is an antichain for all $a\in\mxl(X)$. We will also prove that $\alpha_b=\beta_b=2$ for all $b\in \Relm^{-1}(\mxl(X))=\B$.

Let $a\in\mxl(X)$. Since
$$\sigma_a=\sum\limits_{b\in \Relm^{-1}(a)} \frac{1}{\# \Relm(b)}\geq \frac{1}{2}\#\Relm^{-1}(a)$$
it follows that $\#\Relm^{-1}(a)\leq 3$. Also, $\#\Relm^{-1}(a)\geq 2$ by \ref{rem_relm_geq_2}.

Suppose that $\#\Relm^{-1}(a)=2$. Then, by \ref{lemma_card_B-Ua=2}, $\alpha_b = 3$ for all $b\in \Relm^{-1}(a)$ and thus $\sigma_a = 2$ which entails a contradiction. Hence $\#\Relm^{-1}(a)=3$ and thus $\alpha_b = 2$ for all $b\in \Relm^{-1}(a)$. 

Suppose that $\Relm^{-1}(a)=\{b_1,b_2,b_3\}$. From \ref{lemma_relp_no_S2} it follows that $U_{b_1}$, $U_{b_2}$ and $U_{b_3}$ are pairwise disjoint. Hence, $\Relm^{-1}(a)$ is an antichain and $\beta_{b_j} = 2$ for all $j\in \{1,2,3\}$.

We will prove now that $x>y$ for all $x\in\mxl(X)$ and $y\in\mnl(X)$ and that $\B$ is an antichain. Let $a_1\in\mxl(X)$ and suppose that $\Relm^{-1}(a_1)=\{b_1,b_2,b_3\}$. If $F_{b_1}\cap F_{b_2}\cap F_{b_3}=\varnothing$, it is easy to check that $X-V_{a_1}^X$ is weak equivalent to $S^1$ and hence the splitting triad $(X;X-V_{a_1}^X;V_{a_1}^X)$ satisfies \Spl[2]. Thus, there exists $a_2\in\mxl(X)$ such that $\{b_1,b_2,b_3\}\subseteq U_{a_2}$. It follows that $\Relm^{-1}(a_2)=U_{a_1}\cap\B$ and thus $U_{a_1}\cap\B$ is an antichain. Therefore, $\B$ is an antichain. Now, if $\Relm^{-1}(a_2)=\{b_4,b_5,b_6\}$ then $U_{b_4}$, $U_{b_5}$ and $U_{b_6}$ are pairwise disjoint and hence $\mnl(X) \subseteq U_{b_4}\cup U_{b_5} \cup U_{b_6} \subseteq U_{a_1}$.

Therefore, $\chi(X)=16-(12+12+24)+24=-8$, which entails a contradiction.

\underline{Case 3}: $\# \mxl(X) = 5$.

\underline{Case 3.1}: $\# \mnl(X) = 5$. Thus, $\#\B=6$. By \ref{rem_relm_geq_2}, $\#(U_a\cap\B)\leq 4$ for all $a\in\mxl(X)$. We will prove now that $\#(U_a\cap\B)\leq 3$ for all $a\in\mxl(X)$. Suppose that there exists $a\in\mxl(X)$ such that $\#(U_a\cap\B)=4$. Let $b_1,b_2\in\B$ such that $\B-U_a=\{b_1,b_2\}$. From \ref{lemma_card_B-Ua=2} we obtain that $U_{b_1}\cap U_{b_2}=\varnothing$ and that $\beta_{b_1}\geq 3$ and $\beta_{b_2}\geq 3$, which can not be possible since $\# \mnl(X) = 5$. Therefore, $\#(U_a\cap\B)\leq 3$ for all $a\in\mxl(X)$. Applying the same argument to $X^\op$ it follows that $\#(F_t\cap\B)\leq 3$ for all $t\in\mnl(X)$.

Let $\A=\{x\in\mxl(X)\tq \#(U_x\cap\B)= 3 \}$ and let $\B'$ and $\B''$ be defined as in lemma \ref{lemma_relp_card_B_no_S2}. Let $a\in \A$ and let $b_1$, $b_2$ and $b_3$ be the elements of $\B-U_a$. By \ref{rem_minimal_space_two_points}, $\beta_{b_j}\geq 2$ for all $j\in\{1,2,3\}$. Since $\#\mnl(X)=5$, without loss of generality we may assume that $U_{b_2}\cap U_{b_3}\cap\mnl(X) \neq\varnothing$. Let $z\in U_{b_2}\cap U_{b_3}\cap\mnl(X)$. Then from \ref{lemma_relp_no_S2} we obtain that $\Relp^{-1}(a,z)=\{b_1\}$ and $b_1\in\B'$. If, in addition, $U_{b_1}\cap U_{b_3}\cap\mnl(X) \neq\varnothing$, by the argument above we get that $b_2\in\B'$ and thus $U_{b_1}\cap U_{b_2}\cap\mnl(X) \neq\varnothing$ and it follows that $b_3\in\B'$. Therefore, there exists $a'\in F_{b_1}\cap F_{b_2}\cap F_{b_3} \cap \mxl(X)$. Hence, $a'\in\A$ and thus, by the argument above $\B'-U_{a'}\neq \varnothing$. It follows that $\#\B'\geq 4$ and hence there exists $c\in\mxl(X)$ such that $\#(U_c\cap\B)\geq 4$, which entails a contradiction. Therefore, $U_{b_1}\cap U_{b_3}\cap\mnl(X)=\varnothing$. In a similar way we obtain that $U_{b_1}\cap U_{b_2}\cap\mnl(X)=\varnothing$. And since $\beta_{b_1}\geq 3$, $\beta_{b_2}\geq 2$ and $\beta_{b_3}\geq 2$ it follows that $\beta_{b_1}= 3$, $\beta_{b_2}= \beta_{b_3}=2$ and $U_{b_2}\cap\mnl(X)=U_{b_3}\cap\mnl(X)$. Note that $b_2,b_3\in\B''$. Hence for each $a\in\A$ there exists a unique $b\in\B'$ such that $b\Relm a$.

It follows that $\#\A=\#(\Relm\cap (\B'\times \A))\leq 2 \# \B'$. Therefore
\begin{displaymath}
12 +\frac{\#\A}{2} \leq 12 + \#\B' = 2\#\B+\#\B'=3\#\B'+2\#\B''\leq \sum_{b\in\B}\alpha_b 
\end{displaymath}
and since
\begin{displaymath}
\sum_{b\in\B}\alpha_b \leq 3\#\A + 2(5-\#\A)= 10 + \# \A
\end{displaymath}
we obtain that $12 +\frac{1}{2}\#\A \leq 10 + \# \A$ and hence $\#\A\geq 4$ and $\#\B'\geq 2$.

As in the proof of lemma \ref{lemma_relp_card_B_no_S2}, we have that
$$50\leq 2\sum\limits_{b\in \B'}(5-\alpha_b)(5-\beta_b)+\sum\limits_{b\in \B''}(5-\alpha_b)(5-\beta_b)$$
and since $\#\B'\geq 2$, it is not difficult to prove that $\alpha_b=\beta_b=3$ for all $b\in\B'$ and $\alpha_b=\beta_b=2$ for all $b\in\B''$.

Now, note that $\sigma_a=\frac{7}{6}$ for all $a\in \A$ and since $\sum\limits_{a\in\mxl(X)}\sigma_a=\#\Relm^{-1}(\mxl(X))=\#\B=6$ it follows that $\A\neq\mxl(X)$ and hence $\#\A=4$. Let $c$ be the only element of $\mxl(X)-\A$. Then $\sigma_c=\frac{4}{3}$ and hence $\#\Relm^{-1}(c)\leq 4$. Since $c\notin \A$ we obtain that $\#\Relm^{-1}(c)= 4$ and hence $\alpha_b=2$ for all $b\in\Relm^{-1}(c)$. Thus, $U_c\cap \B=\B'$ and $\#\B'=2$. In a similar way, working with $X^\op$ we obtain that there exists $x\in\mnl(X)$ such that $F_x\cap \B=\B'$.

Let $c_1$ and $c_2$ be the elements of $\B'$. Let $a_1\in\mxl(X)\cap F_{c_1}-\{c\}$. Since $a_1\in\A$, there exist distinct elements $d_1,d_2\in\B''-U_{a_1}$ and by the arguments above we know that $U_{d_1}\cap\mnl(X)=U_{d_2}\cap\mnl(X)$ and $\beta_{d_1}=\beta_{d_2}=2$. Let $y,z\in\mnl(X)$ such that $U_{d_1}\cap\mnl(X)=\{y,z\}$.

Since $\alpha_{d_1}=\alpha_{d_2}=2$ we obtain that there exists $a_2\in(\mxl(X)-\{a_1,c\})\cap F_{d_1}\cap F_{d_2}$. Since $a_2\in\A$, there exist distinct elements $e_1,e_2\in\B''-U_{a_2}$ and by the arguments above we know that $U_{e_1}\cap\mnl(X)=U_{e_2}\cap\mnl(X)$ and $\beta_{e_1}=\beta_{e_2}=2$. Let $v,w\in\mnl(X)$ such that $U_{e_1}\cap\mnl(X)=\{v,w\}$. Note that $\{v,w\}\cap\{y,z\}=\varnothing$ because if $u\in \{v,w\}\cap\{y,z\}$ then $F_u\cap \B \supseteq \{d_1,d_2,e_1,e_2\}$ which entails a contradiction. Note also that $x\notin \{v,w,y,z\}$.

Now take $C=F_x \cup \mxl(X)$ (in fact, $C=F_x$ since it is easy to prove that $\mxl(X)\subseteq F_x$). The connected components of $X-C$ are $\{d_1,d_2,y,z\}$ and $\{e_1,e_2,v,w\}$ and since both of them are included in contractible subspaces of $X$ ($U_{a_2}$ and $U_{a_1}$ respectively) we obtain that $X$ satisfies \Spl[2].

\underline{Case 3.2}: $\# \mnl(X) \geq 6$. From \ref{lemma_relm_card_B_no_S2} we obtain that $\#\B\geq 5$. Thus, $\#\B=5$ and $\#\mnl(X)=6$. Hence, from the proof of \ref{lemma_relm_card_B_no_S2} we obtain that $\sigma_a=1$ for all $a\in\mxl(X)$.

By \ref{rem_relm_geq_2}, $\#\Relm^{-1}(a)\geq 2$ for all $a\in\mxl(X)$. Thus, $\#(U_a \cap\B)\leq 3$ for all $a\in\mxl(X)$.

Suppose that $\#(U_a \cap\B)= 3$ for all $a\in\mxl(X)$. Let $a_1\in\mxl(X)$ and suppose that $U_{a_1} \cap\B = \{b_1,b_2,b_3\}$ and $\B - U_{a_1}=\{b_4,b_5\}$. By \ref{lemma_card_B-Ua=2}, $\{b_4,b_5\}$ is an antichain and $\#(F_{b_4}\cap F_{b_5}\cap\mxl(X))\geq 3$. Let $a_3,a_4,a_5$ be three distinct elements of $F_{b_4}\cap F_{b_5}\cap\mxl(X)$ (note that none of them is $a_1$). Since $\#(U_{a_3} \cap\B)= 3$ it follows that $\#(U_{a_3}\cap \{b_1,b_2,b_3\})=1$. Without loss of generality, we may assume that $b_3\in U_{a_3}$. Hence, $\B - U_{a_3}=\{b_1,b_2\}$ and applying \ref{lemma_card_B-Ua=2} again we obtain that $\#(F_{b_1}\cap F_{b_2}\cap\mxl(X))\geq 3$. Thus, there exists $k\in\{4,5\}$ such that $a_k\in F_{b_1}\cap F_{b_2}$. Hence, $\{b_1,b_2,b_4,b_5\}\subseteq U_{a_k} \cap\B$ which contradicts a previous assertion.

Therefore, there exists $a_1\in\mxl(X)$ such that $\#(U_{a_1} \cap\B)\leq 2$. Then $\# \Relm^{-1}(a_1)\geq 3$ and hence
$$1=\sigma_{a_1}=\sum_{b\in\Relm^{-1}(a_1)}\frac{1}{\#\Relm(b)}\geq \frac{1}{3}  \# \Relm^{-1}(a_1) \geq 1$$
Thus, $\# \Relm^{-1}(a_1)= 3$, $\#(U_{a_1} \cap\B)= 2$ and $\#\Relm(b)=3$ for all $b\in\Relm^{-1}(a_1)$. Hence, $\alpha_b=2$ for all $b\in\Relm^{-1}(a_1)$.

Suppose that $U_{a_1} \cap\B = \{b_1,b_2\}$ and $\Relm^{-1}(a_1)=\{b_3,b_4,b_5\}$. From \ref{lemma_relp_no_S2} it follows that the sets $U_{b_3}$, $U_{b_4}$ and $U_{b_5}$ are pairwise disjoint. In particular, $\{b_3,b_4,b_5\}$ is an antichain. And since $\beta_{b_j}\geq 2$ for all $j\in\{3,4,5\}$ we obtain that $\mnl(X)\subseteq U_{b_3}\cap U_{b_4} \cap U_{b_5}$ and $\beta_{b_j}= 2$ for all $j\in\{3,4,5\}$.

Let $i\in \{1,2\}$. If $\beta_{b_i}\leq 3$ then there exist distinct elements $j,k\in\{3,4,5\}$ such that $\#(U_{b_i} \cap U_{b_j})\leq 1$ and $\#(U_{b_i} \cap U_{b_k})\leq 1$. Let $C=\{b_i,b_j,b_k\}\cup\mnl(X)$. Note that the connected components of $C$ are contractible. Thus, proceeding as in the proof of \ref{lemma_card_B-Ua=2} it follows that the splitting triad $(X;X-C,C)$ satisfies \Spl[2].

Therefore, $\beta_{b_i}\geq 4$ for each $i\in \{1,2\}$. Thus, there exist $w\in\mnl(X)$ such that $w\in U_{b_1}\cap U_{b_2}$. Let $G=F_w\cup\mxl(X)$. Note that the connected components of $X-G$ are contractible. Hence, the splitting triad $(X;X-G,G)$ satisfies \Spl[2].

\underline{Case 4}: $\# \mxl(X) = 6$. In this case, $\#\B \geq 4$ by \ref{lemma_relm_card_B_no_S2}. Hence, $\#\B=4$ and $\#\mnl(X)=6$. Thus, from the proof of \ref{lemma_relm_card_B_no_S2} we obtain that $\sigma_a=\frac{2}{3}$ for all $a\in\mxl(X)$.

Let $a\in\mxl(X)$. By \ref{rem_relm_geq_2}, $\#\Relm^{-1}(a)\geq 2$. If $\#\Relm^{-1}(a)\geq 3$ then $\sigma_a\geq\frac{3}{4}$ which entails a contradiction. Thus, $\#\Relm^{-1}(a)= 2$ for all $a\in\mxl(X)$.

Let $a_1\in\mxl(X)$ and suppose that $U_{a_1}\cap\B=\{b_1,b_2\}$ and $\Relm^{-1}(a_1)=\{b_3,b_4\}$. By \ref{lemma_card_B-Ua=2}, $\{b_3,b_4\}$ is an antichain and $\#(F_{b_3}\cap F_{b_4}\cap\mxl(X))\geq 3$. And since $\sigma_{a_1}=\frac{2}{3}$, we obtain that $\alpha_{b_3}=\alpha_{b_4}=3$ and $\#(F_{b_3}\cap F_{b_4}\cap\mxl(X))= 3$. Let $a_4,a_5,a_6\in\mxl(X)$ such that $F_{b_3}\cap F_{b_4}\cap\mxl(X)=\{a_4,a_5,a_6\}$ and let $a_2$ and $a_3$ be the remaining maximal points of $X$. It follows that $U_{a_j}\cap \B=\{b_3,b_4\}$ for all $j\in\{4,5,6\}$ and hence $U_{a_j}\cap \B=\{b_1,b_2\}$ for all $j\in\{1,2,3\}$ by \ref{rem_minimal_space_two_points}. Note that $\{b_1,b_2\}$ is an antichain by \ref{lemma_card_B-Ua=2}. 

Thus, $X-\mnl(X)$ is homeomorphic to $(\Su D_3)^\op \amalg (\Su D_3)^\op$. Working with $X^\op$ in a similar way, we obtain that $X-\mxl(X)$ is homeomorphic to $\Su D_3 \amalg \Su D_3$.

By \ref{lemma_relp_no_S2}, $U_{b_1}\cap U_{b_2}=\varnothing$ and $U_{b_3}\cap U_{b_4}=\varnothing$. Since $X-\mxl(X)$ is homeomorphic to $\Su D_3 \amalg \Su D_3$ it follows that there exists $k\in\{3,4\}$ such that $U_{b_1}\cap U_{b_k}=\varnothing$. Let $C=U_{b_1}\cup U_{b_k}$. Note that the connected components of $C$ and its complement are contractible. Therefore, the splitting triad $(X;C,X-C)$ satisfies \Spl[2].

Observe the similarities between the proof of this case and the proof of case $m_X=n_X=l_X=4$ of \ref{theo_13_points}.

\underline{Case 5}: $\# \mxl(X) \geq 7$. From lemma \ref{lemma_relm_card_B_no_S2} we obtain that $\#\B \geq m_X.\frac{2}{m_X-3}$ and thus $\#X \geq 2m_X+\frac{2m_X}{m_X-3}>2m_X+2\geq 16$.
\end{proof}

\section{Open questions} \label{section_open_questions}

In this section, we will exhibit finite models for very well-known families of spaces and raise open questions about their minimality.

\subsection{Real projective spaces}

\begin{notat}
  For $n\in \N$ we will write $\Pp^{n}$ for the $n$--dimensional real projective space.
\end{notat}

\begin{prop} \label{prop_models_of_real_projective_spaces}
For $n\in \N$, there exists a finite model of $\Pp^n$ with $\dfrac{3^{n+1}-1}{2}$ points.
\end{prop}
\begin{proof}
Let $\Lambda$ be the finite T$_0$--space whose underlying set is $\{-1,0,1\}$ with the topology associated to the order generated by $0>1$ and $0>-1$.

Let $n\in \N$. Let $S_n$ be the subspace of the product $\Lambda^{n+1}$ obtained by removing its maximum $(0)_{i=1}^{n+1}$. We define the following relation in $S_n$: 
$$(a_i)_{i=1}^{n+1}\sim (b_i)_{i=1}^{n+1}\Leftrightarrow (a_i)_{i=1}^{n+1}=k(b_i)_{i=1}^{n+1} \textnormal{ for $k=1$ or $k=-1$}.$$ 

It is easy to see that $\sim$ is an equivalence relation. Let $P_n=S_n/\sim$ with the topology induced by the canonical projection $q:S_n\to P_n$. It is not difficult to prove that $P_n$ is a T$_0$--space and that $U^{P_n}_{q(a)}=q(U^{S_n}_{a})$ for all $a\in S_n$.

Consider the function $\varphi:\Pp^n\to P_n$ that sends $[(x_i)_{i=1}^{n+1}]$ to $q\left((\sgn(x_i))_{i=1}^{n+1}\right)$. It is clear that $\varphi$ is well defined. We claim that $\varphi$ is a weak homotopy equivalence, and thus $P_n$ is a finite model of $\Pp^n$.

Let $p$ be the canonical projection of $\R^{n+1}-\{0\}$ on $\Pp^n$ and let $\overline{a}=q((a_i)_{i=1}^{n+1})\in P_n$. 
For $1\leq i\leq n+1$ we define
\begin{displaymath}
J_i=\left\{
\begin{array}{ll}
\R&\text{ if $a_i=0$}\\
\R_{>0}&\text{ if $a_i=1$}\\
\R_{<0}&\text{ if $a_i=-1$}\\
\end{array}
\right.
\end{displaymath}
Let $J=\prod\limits_{i=1}^{n+1} J_i$. It is easy to prove that $\varphi^{-1}(U_{\overline{a}})=p(J)$ and that $p(J)$ is contractible and open in $\Pp^n$. By McCord's theorem (\cite[theorem 6]{McC}) it follows that $\varphi$ is a weak homotopy equivalence.
\end{proof}

\begin{rem}
The space $P_1$ constructed in the previous proof is homeomorphic to $\Su S^0$, and hence it is the minimal finite model of $\Pp^{1}\cong S^{1}$. It is easy to check that the space $P_2$ is homeomorphic to the space $\Pp^{2}_1$ of figure \ref{fig_finite_P2_1} which is a minimal finite model of $\Pp^{2}$ by \ref{coro_real_proj_plane}.
\end{rem}

The previous remark leads naturally to the following question.

\begin{question}
Let $n\in \N$, $n\geq 3$. Is the space $P_n$ of the proof of \ref{prop_models_of_real_projective_spaces} a minimal finite model of $\Pp^n$?
\end{question}

\subsection{Tori}

\begin{notat}
For $n\in \N$ we will write $\Tn$ for the $n$--dimensional torus $\prod\limits_{i=1}^{n}S^{1}$.
\end{notat}

\begin{prop} \label{prop_models_of_tori}
Let $n\in \N$. There exists a finite model of $\Tn$ with $4^n$ points.
\end{prop}

\begin{proof}
Let $X=\Su S^0$ which is a finite model of $S^1$ with four points. Then $|\K(X^n)|\simeq |\K(X)|^n\cong \prod\limits_{i=1}^{n}S^{1}\cong \Tn$ and hence $X^n$ is a finite model of $\Tn$ with $4^n$ points.
\end{proof}

\begin{rem} 
Note that $\Su S^0$ is a minimal finite model of $\mathbb{T}^{1}\cong S^{1}$ with four points. Note also that the space $\Su S^0\times \Su S^0$ is homeomorphic to the space $\T_{0,0}$ of figure \ref{fig_finite_Q00} which, by \ref{coro_finite_model_of_torus}, is a minimal finite model of $\T$ with 16 points.
\end{rem}

As above, the previous remark leads to the following question.

\begin{question} 
Let $n\in \N$, $n\geq 3$. Is the space $(\Su S^0)^n$ a minimal finite model of $\Tn$?
\end{question}

\subsection{Moore spaces $M(\Z_n,k)$}

\begin{prop} \label{prop_models_of_Moore_spaces}
Let $n,k\in \N$ with $n\geq 2$. There exists a finite model of $M(\Z_n,k)$ with $4n+2k+3$ points.
\end{prop}

\begin{proof}
Let $C_n$ be a regular $2n$--gon with vertices $c_0,\ldots,c_{2n-1}$ and center $d$. Let $c_{2n}=c_0$. We regard $C_n$ as regular CW-complex with $2$--cells $c_t c_{t+1} d$ for $0\leq t\leq 2n-1$, $1$--cells $c_t d$ and $c_t c_{t+1}$ for $0\leq t\leq 2n-1$ and $0$--cells $d$ and $c_t$ for $0\leq t\leq 2n-1$. 

We obtain a CW-model of $M(\Z_n,1)$, which will be denoted $X_n$, by identifying (directed) edges $c_{2t}c_{2t+1}$ for $0\leq t\leq {n-1}$ and  $c_{2t-1}c_{2t}$ for $1\leq t\leq {n}$.

Now, $X_n$ inherits from $C_n$ a regular CW-complex structure which has $2n$ $2$--cells, $2n+2$ $1$--cells and $3$ $0$--cells. Let $Z_n$ be the face poset of $X_n$. It follows that $Z_n$ is a finite model of $M(\Z_n,1)$ which has $4n+5$ points. Hence $\Su^{k-1} Z_n$ is a finite model of $M(\Z_n,k)$ with $4n+2k+3$ points. 
\end{proof}

\begin{rem}
The space $X_2$ constructed in the proof above is homeomorphic to $\Pp^{2}$. Moreover, it has the regular CW-complex structure of figure \ref{fig_CW_P2_1} and thus, the space $Z_2$ is homeomorphic to the space $\Pp^{2}_1$ of figure \ref{fig_finite_P2_1}. By \ref{coro_real_proj_plane}, $Z_2$ is a minimal finite model of $M(\Z_2,1)$.
\end{rem}

\begin{question} 
Let $n,k\in \N$ with $n\geq 2$. If $n\geq 3$ or $k\geq 2$, is the space $\Su^{k-1} Z_n$ a minimal finite model of $M(\Z_n,k)$?
\end{question}

We finish this article stating a conjecture which is related to the previous question and to the conjecture of Hardie, Vermeulen and Witbooi that we proved in this work (cf. \ref{theo_HVW}).

\begin{conjecture}
Let $X$ be a finite $T_0$--space and let $k\in\N$ such that $k\geq 2$. If $H_k(X)$ has torsion then $\# X \geq 2k+11$.
\end{conjecture}

Clearly, if this conjecture is true then the finite models of $M(\Z_2,k)$ constructed above will be minimal finite models and thus we will obtain an affirmative answer to the case $n=2$ of the previous question.

\end{document}